\documentclass{article}

\usepackage{etex}
 \usepackage{geometry} 
 \usepackage{appendix}
\usepackage{rotating}
 \usepackage[T1]{fontenc}
\usepackage{graphicx}
\usepackage{tabularx}
\usepackage{amsfonts}
\usepackage{amsthm}
\usepackage{amsmath}
\usepackage{proof}
\usepackage{mathdots}
\usepackage{hyperref}
\usepackage{xcolor}
\usepackage{amssymb}
\usepackage{color}
\usepackage{tikz}
\usepackage{mathrsfs}
\usepackage[all]{xy}
\usepackage{lscape}
\usepackage{bussproofs}
\usepackage{fge}

\EnableBpAbbreviations

\theoremstyle{plain}

\newtheorem{theorem}{Theorem}[section]

\newtheorem{lemma}[theorem]{Lemma}

\newtheorem{proposition}[theorem]{Proposition}

\numberwithin{equation}{section}



\newcommand{\B}[1]{\mathbf{#1}}
\newcommand{\TT}[1]{\mathtt{#1}}
\newcommand{\D}[1]{\mathscr{#1}}

\newcommand{\C}[1]{\mathcal{#1}}

\newcommand{\OV}[1]{\overline{#1}}

\newcommand{\To}{\Rightarrow}

\definecolor{color0}{HTML}{4682B4}

\title{Polymorphism and the obstinate circularity \\ of second order logic: a victims' tale}

\author{Paolo Pistone}
\date{}

\begin{document}
\maketitle

\begin{abstract} The investigations on higher-order type theories and on the related notion of parametric polymorphism constitute the technical counterpart of the old foundational problem of the circularity (or impredicativity) of second and higher order logic. However, the epistemological significance of such investigations, and of their often non trivial results, has not received much attention in the contemporary foundational debate.

The results recalled in this paper suggest that the question of the circularity of second order logic cannot be reduced to the simple assessment of a vicious circle. Through a comparison between the faulty consistency arguments given by Frege and Martin-L\"of, respectively for the logical system of the \emph{{Grundgesetze}} (shown inconsistent by Russell's paradox) and for the intuitionistic type theory with a type of all types (shown inconsistent by Girard's paradox), and the normalization argument for second order type theory (or System $\B F$), we indicate a bunch of subtle mathematical problems and sophisticated logical concepts hidden behind the hazardous idea of impredicative quantification, constituting a vast (and largely unexplored) domain for foundational research.
\end{abstract}

\section{Introduction}

It is usually agreed that the history of second order logic starts with Frege's pioneering work, especially in the \emph{Begriffschrift} and in the \emph{Grundgesetze}. In the first text, the German logician introduced a logical syntax with a clear distinction between ``first-level'' and ``second level'' functions. In the second, he introduced a logical semantics centered around the notion of denotation (\emph{Bedeutung}), along with a ``proof'' of what we would call nowadays a soundness theorem, stating that every expression in his system has a (unique) denotation. 

The latter result obviously implied the consistency of Frege's system, as it stated that with every proposition there is associated a unique truth-value. Everybody knows the end of this story: in 1902 Russell constructed an antinomy in Frege's system, that is, a proposition which becomes false as soon as it is true and becomes true as soon as it is false. Not only Russell had contradicted Frege's result that every proposition is associated with a unique truth-value, but he had explicitly constructed a contradiction in the system. In a word, Frege's system was inconsistent and his alleged proof was wrong.

Many years later, in 1971, Martin-L\"of introduced a very general syntax for type theory based on a strongly impredicative axiom, stating the existence of a type of all types. The Swedish logician also provided a semantics for this system centered around the notion of ``computability'', an elegant generalization of previous work by Tait and Girard. Actually, in those years several results in proof theory were dissolving the belief that no constructive approach to higher-order logic were possible. Martin-L\"of's system pursued then the ambitious goal of providing an abstract and uniform formalization of higher-order logic within intuitionistic type theory.

Martin-L\"of's original unpublished paper (\cite{ML70}) contained an alleged proof that all terms in his system are computable. This result obviously implied the consistency of the system, as it stated that, with every well-typed term of the system, there is associated a
unique normal form (and no well-typed term in normal form for the absurdity exists). Unfortunately, history often repeats itself: in the same year, Girard showed that a variant of Burali-Forti's paradox could be typed in Martin-L\"of's system, producing a term having no normal form. Not only Girard had contradicted Martin-L\"of's result that every term has a unique normal form, but he had explicitly constructed a contradiction in the system. In a word, Martin-L\"of's system was inconsistent and his alleged proof was wrong.

Very instructive insights into the much debated issue of the circularity (or impredicativity) of second order logic are provided by these two surprisingly similar, and equally unfortunate, episodes about such two great logicians. It is often argued that Frege's consistency argument succumbed to what Russell called the vicious circle principle of second order logic, that is the fact that the second order quantifier quantifies over a class of propositions including those which are defined by means of that very operator. In the case of Martin-L\"of's argument the circularity in question is more difficult to identify, as no apparent vicious circle occurs in the proof.

Nowadays, the vicious circles of impredicative reasoning, far from being eradicated from the logical world, appear not only in mathematics, but also in the instructions of many programming languages (under the names of polymorphism, generic programming, templates). From a theoretical question concerning the foundations of mathematics, the philosophical problem of impredicativity has evolved into a technical issue in the design of abstract and uniform programming tools.

As is well-known, a substantial symmetry between \emph{methods of proof} and \emph{methods of programming} was revealed by the so-called \emph{propositions-as-types} or \emph{Curry-Howard} (\cite{Curry1958, Howard80}) correspondence between natural deduction systems and typed $\lambda$-calculi. It is such a symmetry which makes it possible that the normalization theorem for System $\B F$ (a typed -calculus implementing second order polymorphism), whose proof constitutes the \emph{trait d'union} of our reconstruction, has a double interpretation: as a result in computer science, it provides the grounds for our confidence that the execution of higher-order, impredicative, programs, ultimately halts; as a result in logic, it warrants the consistency of second order logic in a very strong, proof-theoretical, sense.

This theorem is just one of a series of foundational results (like Reynolds' theory of parametric polymorphism or the denotational interpretation of polymorphism) that had a significative echo in proof theory and theoretical computer science, but went practically unnoticed in the philosophical debate on impredicativity. It seems, however, hard to deny that an up-to-date evaluation of the old question of impredicativity passes through an epistemological analysis of these advances.

In our reconstruction a characterization is attempted of different forms of circularity, which Russell's 110 years old diagnosis seems incapable to distinguish. In particular, a distinction is made between the circularity of which Frege was victim: that of establishing the truth of a proposition which refers to all propositions, and the circularity of which Martin-L\"of was victim: that of establishing that an impredicative rule is sound by employing, in the argument, occurrences of a similarly impredicative rule.

This paper is intended both as a survey of some important results in the proof-theory of second order logic and polymorphism, which have not received much attention in
the philosophical debate, and as a starting point for a discussion of their foundational significance.

\section{A programmer in the barber's shop.}\label{sec::1}
 The problem of circularity, or impredicativity, is here historically introduced by recalling Russell's diagnosis of the antinomies, based on the individuation of a specific form of circularity, namely the violation of the \emph{vicious circle principle}. Type theory was introduced by Russell to let propositions obey the vicious circle principle and to prohibit self-application. However the developments of type theory, which had more success in theoretical computer science than in foundations, betrayed the foundational concerns that had led to its introduction: polymorphic type disciplines reintroduce impredicativity and self-applications in a typed-setting. In particular, the \emph{Curry-Howard correspondence}, here briefly recalled, shows that polymorphism and higher-order quantification are actually two ways to look at the same phenomenon.

The problem of impredicativity is decomposed into a \emph{metaphysical problem}, related to the existence of impredicatively defined entities, and an \emph{epistemic problem}, related to the possibility of ascertaining the truth or proving an impredicative proposition. Concern-ing the latter, two forms of circularity are retrieved (to which section 2 and section 3 will be devoted, respectively): a circularity \emph{in the definition} of the truth or provability conditions for impredicative propositions (related to the problem of providing meanings or denotations to higher-order propositions), and a circularity \emph{in the elimination} of redundant impredicative concepts in a proof (related to the problem of establishing the \emph{Hauptsatz} for higher-order logic).

\subsection{Russell's \emph{vicious circle principle}.} A quite general way to express the prob-lem exposed by the antinomies in Frege's logical system and in na\"ive set theory is the following: in every logical system in which sets can be defined by means of any clause of the form \emph{the set of all sets such that $\varphi$} where $\varphi$ is a property of sets, an antinomy, that is, a sentence logically equivalent to its negation, can be derived\footnote{Remark that, if $s$ is the set of all sets such that $\varphi$, then asking whether $s$ satisfies $\varphi$ is equivalent to asking whether $s$ belongs to $s$.}.

Since in standard (minimal, intuitionistic or classical) logics, a contradiction can be derived from an antinomy\footnote{However, this is not the case for some logics like \emph{Elementary Linear Logic} and \emph{Light Linear Logic} (\cite{Girard1996}), in which structural rules are controlled in a special way (for reasons connected with the theory of implicit computational complexity), giving rise to different analyses of the paradoxes.}, one must conclude that the appeal to definitions of this form has to be restricted in some way.

As is well-known, Russell (\cite{Russell1906}) and Poincar\'e (\cite{Poincare1906}) suggested that \emph{all} such definitions should be forbidden, as they violate the so-called \emph{vicious circle principle} $\B{VCP}$, stated by Russell as follows:

\begin{quote}\small
Whatever in any way concerns all or any or some of a class must not itself be one of the members of that class. \cite{Russell1906}
\end{quote}

In \cite{Godel1944} G\"odel observes that the core of the debate on the $\B{VCP}$ actually lies in a special case of the vicious circle principle, namely the case in which one \emph{defines} a set by quantifying over a class of which the new defined set is a member. In particular, it is this special case which forbids to define sets by means of clauses of the form \emph{the set of all
sets such that $\varphi$}, as the definition makes reference to the totality of sets, hence including the set here defined. We can restate thus the $\B{VCP}$ as follows:

\begin{quote}\small
Whatever can only be defined in terms of all or any or some of a class must not itself be one of the members of that class
\end{quote}

We will call \emph{impredicative} a clause defining a given notion and violating the $\B{VCP}$.

On the one side, the $\B{VCP}$ prevents constructions like those leading to Russell's para-dox, as it forbids definitions of the form mentioned above. On the other side, this principle forbids basic second order constructions:

\begin{description}
\item[$i.$] the definition of a proposition by quantification over all propositions: for instance, the propositions \emph{every proposition is either true or false} and \emph{some proposition implies its own negation}; this is explicitly observed in \cite{Russell1910}:

\begin{quote}\small
The vicious circles in question arise from supposing that a collection of objects may contain members which can only be defined by means of the collection as a whole. Thus, for example, the collection of propositions will be supposed to contain a proposition stating that ``all propositions are either true or false''. It would seem, however, that such a statement could not be legitimate unless ``all proposition'' referred to some already definite collection, which it cannot do if new propositions are created by statements about ``all proposition''. We shall, therefore, have to say that statements about ``all propositions'' are meaningless. \cite{Russell1910}
\end{quote}

\item{$ii.$} the definition of a set as the intersection of all sets satisfying a certain property: for instance the intersection of all infinite sets, that we can equivalently call the smallest infinite set. This is remarked for instance in \cite{Godel1944},
\begin{quote}\small
[...] if one defines a class as the intersection of all classes satisfying a certain condition $\varphi$ and then concludes that is a subset also of such classes u as are defined in therms of (provided they satisfy $\varphi$). \cite{Godel1944}
\end{quote}
\end{description}

We will call impredicative a proposition containing a quantification over all propositions as well as any set defined as the intersection of all sets satisfying a given property.
Hence, if one endorses the $\B{VCP}$ as a way-out of the paradoxes, and in particular if one endorses the view that any legitimate logical system must obey the $\B{VCP}$, then he must count second order logic as not legitimate as logic, as it allows for ``viciously circular'' definitions.

It must be stressed that the validity of the $\B{VCP}$ is not implied by the existence of the antinomies. For instance, in the standard set theory $\B{ZF}$, a theory devised in order to avoid the paradoxes of na\"ive set theory, definitions violating the $\B{VCP}$ are possible: let $\varphi$ be any property stable by intersection\footnote{That is, such that, if $(a_{i})_{i\in I}$ is a collection of sets indexed by a set $I$, and $\varphi(a_{i})$ holds for all $i\in I$, then $\varphi(\cap_{i\in I}a_{i})$ holds. One can take as $\varphi$ the property of being a subset of a given set $t$.
} given a set $s$, one can define a new set $s'$ as the smallest subset of $s$ satisfying $\varphi$, i.e. as the intersection of all subsets of $s$ satisfying $\varphi$. Then $s'$, which belongs to the collection of sets (the subsets of $s$), to which its definition makes reference, satisfies $\varphi$.

What form should then a logical system in accordance with the $\B{VCP}$ have? In a later text (\cite{Russ08}) Russell expressed the $\B{VCP}$ in a more syntactic way as follows:

\begin{quote}\small
Whatever contains an apparent variable must not be a possible value of that variable. \cite{Russ08}
\end{quote}

If $s$ is the set of all sets such that $\varphi$, then in the definition of s a bound variable occurs, standing for elements of the range of $\varphi$, i.e. sets, of which s is a possible value. In order to avoid such constructions, Russell proposed the introduction of a \emph{type discipline}. All expressions in the language are assigned a type, following the two principles below:

\begin{description}

\item[T1:] the range of application of a function forms a type;

\item[T2:] if the range of a function $f$ is $\sigma$, then $f$ has a higher type $\sigma\to \tau$ , where $\tau $ is the type of the image of $f$.

\end{description}

With Russell's type disciple, propositions, that is, those expressions which are either true or false, are simply those expressions which receive the type $prop$, the type of propositions.

The type discipline given by $\B{T1-2}$ prevents the application of a function to itself: if a function has range $\sigma$, then it must have a type of the form $\sigma\to \tau$, hence it cannot be among its possible values. Similar considerations for sets can be made, as soon as one identifies a set with the property defining it (which is a function from a given range to truth values): as the function cannot belong to its own range, a set cannot belong to itself.

However, a type-discipline in accordance with the principles $\B{T1-2}$ in which the second order construction $i.$ is well-typed (that is, where an impredicative proposition $\forall XA $ is given type $prop$) can be devised, by stating that, if $A[X]$, where $X$ has type $prop$, is a function of type $prop \to prop$, then $\forall XA$ has type $prop$ (that is, by stating that the second order quantifier $\forall$ is a function of type $(prop \to prop) \to prop$); this discipline, usually called \emph{simple type theory} (\cite{Russell1903}, see next section), is compatible with $\B{T1-2}$ but violates the $\B{VCP}$: the proposition $\forall XA$ appears among the possible values of $X$.

A more rigid type discipline was introduced by Russell in order to forbid impredicative quantification: rather than considering a unique type $prop$ for propositions, he introduced a hierarchy of types $prop_{n}$ for propositions of different complexities, and stated that, if $A[X]$ has type $prop_{n} \to prop$, then $\forall^{n}XA$ has type $prop_{n+1}$, i.e. it is a proposition of strictly higher complexity (that is, for any $n$, the quantifier $\forall^{n}$ is of type $(prop_{n} \to prop) \to prop_{n+1}$). This discipline, called \emph{ramified type theory} (\cite{Russ08}), does not violate the $\B{VCP}$, since, in the proposition $\forall ^{n}XA$, $\forall^{n}XA $ itself does not appear among the possible values of $X$.

The propositions typable in simple type theory yield a quite expressive system, including second order logic. This means that, even if one accepts, as Russell did, that the $\B{VCP}$ forces to accept the type discipline $\B{T1-2}$, the converse does not hold: the type discipline T1 2 is compatible with the existence of impredicative types. Ramified type theory yields a much less expressive system, which does not include full second order logic\footnote{In particular, it is incapable of representing quite basic arithmetical operations like exponentiation.
}. As is well-known, Russell and Whitehead's solution in the Principia was a com-promise between the two: a ramified type theory with the \emph{axiom of reducibility}, which implies that, for any proposition of type $prop_{n}$ there exists an equivalent proposition of type $prop_{0}$.

\subsection{Polymorphism: the ``non-Russellian'' type discipline of second order proofs.} The success of Russell's type discipline as a new foundation for mathematics
(and especially set-theory) was quite limited. As we saw, the dominant foundational theory, $\B{ZF}$, is not in accordance with the $\B{VCP}$. However, his ideas had a significant appeal in the development of theoretical computer science, as modern type theory constitutes the foundations of the theory of programming languages. At the same time, consistent type disciplines allowing for impredicative types and self-application of functions were introduced through the notion of polymorphism. In a word, type theory progressively rejected the foundational concerns which had been the very reason of its introduction.

A central position in the path leading from Russell's type theory to modern type theories (a complete reconstruction of this path obviously exceeds the purposes of this paper - the reader can look for instance at \cite{Kamare2004}) is occupied by the development of \emph{typed $\lambda$-calculi} as abstract functional programming languages.

In the \emph{simply typed $\lambda$-calculus} $\lambda_{\to}$, first proposed by Church (\cite{Church}), terms are defined in accordance with the type discipline $\B{T1-2}$\footnote{We recall that $\lambda_{\to}$ terms, also called $\lambda$-terms, are generated, starting from a collection, for any type $\sigma$, of variables $x, y, z\in \sigma$ of that type, by two operations: \emph{abstraction}, which, given a term $M^{\tau}$ and a variable $x\in\sigma$, forms a term $(\lambda x\in \sigma)M^{\tau}$ , of type $\sigma\to \tau$; \emph{application}, which, given a term $M^{\sigma\to \tau}$ and a term $N^{\sigma}$ , forms a term $M^{\sigma\to \tau} N ^{\sigma}$, of type $\tau$. Moreover, the execution of terms is governed by the
reduction rule $((\lambda x\in \sigma)M^{\tau})N^{\sigma} \ \to \  M^{\tau}[N^{\sigma}/x\in \sigma]$, which computes the result of applying a function constructed by the abstraction operation.
}. 
In $\lambda_{\to}$ no well-typed program can be applied to itself: if $M$ is a program of type $\sigma$, then, in order for the application $MM$ to be well-typed, one should also have that $M$ has type $\sigma\to \tau$, for some type $\tau$, which is impossible, since programs are assigned a unique type. Hence, Russell's ban on self-application has a direct translation in $\lambda_{\to}$.

The fact that programs come with a unique type is sometimes referred to as static typing. Contrarily to Russell's type discipline, in modern programming languages there exist programs which can be assigned more than one (actually, infinitely many) types: this \emph{polymorphism} of programs gives rise to type disciplines which, while appearing, at a first glance, quite natural from a programming point of view, are extremely challenging from the point of view of logic.

The natural aspect of polymorphism is shown by the following example: suppose we want to write a program $\TT{ID}$ which, when applied to an arbitrary other program $\TT P$, gives as output the program $\TT P$. The behavior of $\TT{ID}$ is then governed by the simple equation

$$\TT{ID\  P} \ =\  \TT P$$

To represent this very simple program in $\lambda_{\to}$, we are forced to make an apparently arbitrary choice, namely the one of the range of $\TT{ID}$. In other words, we are forced to restrict the programs to which $\TT{ID}$ can be applied to those belonging to a given type . Hence, for any choice of a type $\sigma$, a different coding of our program in $\lambda_{\to}$ must be given, by the $\lambda$-term $\TT{ID}_{\sigma} = (\lambda x\in \sigma)x$, of type $\sigma\to \sigma$. However, the program is so simple that it seems natural to look for a uniform coding of $\TT{ID}$ such that, for any type $\sigma$, we can say that $\TT{ID}$ can be given the type $\sigma\to \sigma$.

It seems natural then to look for programs which are, so to say, so simple that they can work \emph{at any type}. Let us make the hypothesis that a uniform and polymorphic coding of $\TT{ID}$ exists, and let us fix a type $\sigma_{0}$; as $\TT{ID}$ can be given type $\sigma\to \sigma$ for any type $\sigma$, then it can be given the type $\sigma_{0}\to \sigma_{0}$, but also the type $(\sigma_{0}\to \sigma_{0})\to (\sigma_{0}\to \sigma_{0})$; this means that one of the types of $\TT{ID}$ is actually one of the possible ranges of $\TT{ID}$, which
authorizes the application $\TT{ID \ ID}$ of $\TT{ID}$ to itself. In other words, as soon as one follows the intuition that a simple enough program can be given infinitely many types, Russell's ban on self-applications is violated. Actually, the polymorphic intuition is at work in most common programming languages, pace Russell.

A first, informal, description of this phenomenon appeared in a 1967 seminal text by Strachey, in which the notion of parametric polymorphism (i.e. the fact that a term can have infinitely many types) was first introduced:

\begin{quote}\small
Parametric polymorphism [...] may be illustrated by an example. Suppose $\TT f$ is a function whose argument is of type $\alpha$ and whose results is of type $\beta$ (so that the type of $\TT f$ might be written $\alpha\to \beta$), and that $\TT L$ is a list whose elements are all of type $\alpha$ (so that the type of $\TT L$ is $\alpha\TT{list}$). We can imagine a function, say $\TT{Map}$, which applies $\TT f$ in turn to each member of $\TT L$ and makes a list of the results. Thus $\TT{Map}[\TT f, \TT L]$ will produce a $\beta\TT{list}$. We would like $\TT{Map}$ to work on all types of lists provided $\TT f$ was a suitable function, so that $\TT{Map}$ would have to be polymorphic. However its polymorphism is of a particularly simple parametric type which could be written
$$(\alpha\to \beta, \alpha\TT{list})\to \beta\TT{list}$$
where $\alpha$ and $\beta$ stand for any types. \cite{Strachey1967}
\end{quote}

In definitive, a polymorphic program is one which can be given an infinite class of types, all of the form $\sigma[\tau/X]$, where $\sigma$ is an opportune type and varies among all types.

Strachey's ideas were formalized in 1971 by Girard (\cite{Girard72}) (and, independently, in 1974 by Reynolds \cite{Reynolds74}) with the introduction of System $\B F$, or polymorphic $\lambda$-calculus. The type discipline of System $\B F$ differs from Russell's one for the fact that types contain variables, and a variable intuitively stands \emph{for all} its substitution instances. Hence when a function is assigned a range $\sigma$ and the variable $X$ occurs freely in $\sigma$, this means that the function can be assigned all types $\sigma[\tau/X]$, for any type $\tau$. To the clauses $\B{T1-2}$ a third and a fourth clause are added:

\begin{description}
\item[T3:] if a function can be assigned the type $\sigma$, where the variable $X$ occurs freely in $\sigma$\footnote{And moreover X does not occur free in any open assumptions.}, then it can be assigned the type $\forall X\sigma$ ;

\item[T4:] if a function has type $\forall X\sigma$ then, for all type $\tau$, it has also type $\sigma[\tau/X]$.
\end{description}

More precisely, in addition to the operations of abstraction and application of $\lambda_{\to}$, two new operations are present in System $\B F$: \emph{type abstraction}, which, given a term $M^{\sigma}$ of type $\sigma$ and a type variable $X$ occurring freely in $\sigma$\footnote{And moreover X does not occur free in any other type declaration.
}, forms a term $\Lambda X.M^{\sigma}$, of type $\forall X\sigma$; \emph{type extraction} which, given a term $M^{\forall X\sigma}$ of type $\forall X\sigma$ and a type $\tau$, forms a term $M \{\tau\}$, of type $\sigma[\tau/X]$. To the reduction law governing the execution of programs
a new reduction law is added: $(\Lambda X.M^{\sigma})\{\tau\} \ \to \ M[\tau/X]^{\sigma[\tau/X]}$, which computes the result of extracting over a type a program previously abstracted over all types\footnote{It's worth mentioning that System $\B F$ (as well as $\lambda_{\to}$) can be presented in the so-called \emph{\`a la Curry} version, as a type discipline over untyped $\lambda$-terms. Untyped $\lambda$-terms are defined similarly to the typed ones, but without type annotations. Starting from untyped variables $x,y,z,\dots$, the abstraction operation allows to form a term $\lambda x.M$ starting from a term $M$ and a variable $x$; the application operation allows to form a term $M N$ starting from two arbitrary terms $M$ and $N$. In the untyped case, the typing conditions for System $\B F$ do not affect types. In the case of type abstraction, if $M$ is a term of type , then $M$ itself is a term of type $\forall X\sigma$ (provided the variable $X$ does not occur free in all the type declarations for the free variables of $M$); in the case of type extraction, if $M$ is a term of type $\forall X\sigma$, then, for all type $\tau$, $M$ itself is a term of type $\sigma[\tau/X]$. In this system (which is equivalent to the typed version) polymorphism appears in a more literal way: the same untyped $\lambda$-term might have infinitely many types.\label{footnote::8}
}.

The definition of impredicative types of the form $\forall X\sigma$ must be distinguished from the fact of typing impredicative propositions: while the latter can be done in simple type theory, the former requires a proper extension of the simple type discipline. One can then ask what theory results by considering those propositions which can be well-typed in a polymorphic (say System $\B F$) type discipline: this hugely impredicative theory, called System $\B U^{-}$ in \cite{Girard72}, is much more expressive than System $\B F$ alone and will be discussed in section 3, in connection with Martin-L\"of's impredicative type theory.

The program $\TT{ID}$ has a unique representation in System $\B F$, by means of the $\lambda$-term $\Lambda X.(\lambda x\in X)x$ which has type $\forall X(X\to X)$. The type-dependent representations of $\TT{ID}$ in $\lambda_{\to}$ are easily obtained from $\TT{ID}$ by type extraction: $\TT{ID}_{\sigma} = \TT{ID}\{\sigma\}$. Moreover, by performing two different extractions, the term $\TT{ID}$ can be correctly applied to itself, yielding the program $(\TT{ID}\{X\to X\})\TT{ID}\{X\}$\footnote{In the version \emph{\`a la Curry}, the program $\TT{ID}$, as well as all programs $\TT{ID}_{\sigma}$, is represented by the untyped $\lambda$-term $\lambda x.x$. Moreover, the term $(\lambda x.x)\lambda x.x$, corresponding to the self-application of $\TT{ID}$ to itself, can be well-typed.
}.

Similarly, the program $\TT{Map}$ mentioned by Strachey has a unique representation, of type $\forall X\forall Y((X\to Y)\to (X\TT{list}\to Y\TT{list}))$\footnote{where $X\TT{list}$ is the type $\forall Y(Y\to (X\to Y\to Y)\to Y)$whose elements are either the empty list $\epsilon= \Lambda Y.(\lambda x\in Y)(\lambda c\in X\to Y\to Y)x$ or the concatenation $\TT{conc}(\TT l, \TT a)= \Lambda Y.(\lambda x\in Y) \lambda c\in X\to Y\to Y) c \TT a((\TT l\{Y\})xc)$ of an object $\TT l$ of type $ X\TT{list}$ and an object $\TT a$ of type $X$.
}. Also the program $\TT{Map}$ can be applied to itself, as soon as one chooses the type $(X\to Y)$ and $(X\TT{list}\to Y\TT{list})$ as values, respectively, of $X$ and $Y$.

Though this type discipline allows for constructions which clearly violate the $\B{VCP}$, paradoxical constructions like Russell's paradox cannot be typed in System $\B F$. In particular, Girard showed that, similarly to the much weaker calculus $\lambda_{\to}$, the execution of any program in System $\B F$ reduces to a unique normal form (Girard's normalization theorem will be reconstructed in detail in the next section). As we will discuss in the next subsection, from a logical point of view, this result establishes the consistency of second order logic (and of second order arithmetics) in a very strong proof-theoretic sense.

\subsection{The epistemic problem of impredicativity.} In the debate on the $\B{VCP}$ two different, though related, problems can be distinguished: the problem concerning the existence of the objects defined by impredicative clauses, that we will call the \emph{metaphysical problem}; the problem concerning the ways to prove that some objects belong to a set defined by an impredicative clause, or that some impredicative proposition is true, that we will call the \emph{epistemic problem}.

The metaphysical problem can be stated as follows: how can an impredicatively defined set be said to exist, if its definition presupposes the existence of a totality of sets, con-taining the set it defines? The answer seems to depend upon the philosophical position about the nature of mathematical objects one adopts. For instance, a problem of course occurs if one endorses the view that mathematical objects owe their existence (if any) to their definition, either intended in purely formalistic way, or intended as a representation of some form of (mental or linguistic) construction. If, however, as G\"odel remarks,

\begin{quote}\small
it is a question of objects that exist independently of our constructions, there is nothing in the least absurd in the existence of totalities containing members, which can be described (i.e. uniquely characterized) only by reference to this totality. \cite{Godel1944}
\end{quote}

Hence, for instance, no problem of existence would arise from a platonistic standpoint. The analysis of impredicativity in the pages that follow will not take into consideration the metaphysical problem. Our focus is indeed directed towards some the issues which arise in the attempt at defining truth-conditions or proof-conditions for impredicative propositions.  These issues are captured by the epistemic problem, which arises from the remark that, when one argues for the truth of a proposition quantifying over all propositions, or for the fact that some set belongs to the set of all sets such that $\varphi$, one might enter an infinite regress.  For instance, if one wants to know whether the proposition $P$ : \emph{every proposition is either true or false} is true or false, one has to verify, for every proposition, whether it is true or false, hence in particular one has to verify this
for the proposition $P$, ending in a regress.

In a famous paper (\cite{Poincare1906}), Poincar\'e identifies a circularity in reasoning about impredicative concepts starting from a broad conception of logic as an essentially ``sterile'' activity: the French mathematician argues that a purely logical proof is one which can be trans-formed into a series of tautological propositions, once the expressions involved in it are replaced by their definitions. This is in contrast, for the author, with mathematical proofs, which do not reduce to tautologies but to propositions the acknowledgement of whose truth requires an appeal to intuition. For instance, he insists that, if one has logically proved an equality of the form $a = b$ , then he must be able to reduce this equality, by progressively substituting concepts by their definition, into tautological equalities of the form $c = c$.

\begin{quote}\small
Mais si l'on remplace successivement les diverses expressions qui y figurent par leur d\'efinition et si l'on poursuit cette op\'eration aussi loin que l'on le peut, il ne restera plus \`a la fin que des identit\'es, de sorte que tout se r\'eduira \`a une immense tautologie. La Logique reste donc st\'erile, \`a moins d'\^etre f\'econd\'ee par l'intuition. \cite{Poincare1906}
\end{quote} 

Once established this, Poincar\'e argues that the replacement of an impredicatively defined concept with its definition might fail to produce a series of tautologies. For instance, the truth of the proposition $P$ above can be defined by stating that $P$ is true when, for any proposition $Q$, $Q$ is either true or false. Now, to ascertain whether $P$ is true, we substitute the proposition ``$P$ is true'' by its definition, and we find the proposition ``$P$ is either true or false'' among all possible $Q$, ending in a regress which prevents the reduction to a series of tautologies. Poincar\'e, notoriously, concludes that:

\begin{quote}\small
Dans ces conditions, la Logistique n'est plus st\'erile, elle engendre l'antinomie. \cite{Poincare1906}
\end{quote}

Poincar\'e's argument is sometimes recalled in order to show that impredicative propositions must be considered viciously circular from a constructivist standpoint. For instance, Sundholm refers to this argument when explaining why, in his opinion, the meaning conditions for impredicative propositions are circular:

\begin{quote}\small
A meaning-explanation for the second order quantifier begins by stipulating that $(\forall X : Prop)A$ has to be a proposition under the assumption that $A$ is a propositional function from $Prop$ to $Prop$, that is, that $A : Prop$, provided $X : Prop$. One then has to explain, still under the same assumption, which proposition it is:

$$(\forall X : Prop)A\text{ is true if and only if }A[P = X]\text{ is true, for each proposition }P$$

In the special case of $(\forall X : Prop)X$ one obtains

$$(\forall X : Prop)X\text{ is true }=_{def}\text{ $ P$ is true, for each proposition } P$$
but $(\forall X : P rop)X$ is (meant to be) a proposition, so it has to be considered on the righthand side. Accordingly (4.2.2) cannot serve as a definition of what it is for $(\forall X : P rop)X$ to be true; it does not allow for the elimination, effective or not, of ...is true when applied to the alleged proposition $(\forall X : Prop)X$. \cite{Sundholm1999}
\end{quote}

A different view is advocated by Carnap (a well-known defender of impredicativity), who remarks that, if we really endorse the ``sterility'' view of logic and reject impredicativity in proofs, then we are forced to reject even basic arithmetical proofs. Indeed, an infinite regress arises even when one tries to argue that the number three is inductive (i.e. that for all property P which holds of zero and holds of $n + 1$ as soon as it holds of $n$, $P$ holds of three):

\begin{quote}\small
For example, to ascertain whether the number three is inductive, we must, according to the definition, investigate whether every property which is hereditary and belongs to zero also belongs to three. But if we must do this for every property, we must also do it for the property ``inductive'' which is also a property of numbers. Therefore, in order to determine whether the number three is inductive, we must determine among other things whether the property ``inductive'' is hereditary, whether it belongs to zero and, finally - this is the crucial point - whether it belongs to three. But this means that it would be impossible to determine whether three is an inductive number. \cite{Carnap1983}
\end{quote}

A dilemma seems to come out of this argument: either one accepts that even some basic arithmetical arguments (like the inductive proof that three is a natural number) are flawed, since circular, or one must admit that the intuition that ascertaining the truth (or proving) an impredicative proposition amounts at ascertaining the truth (or proving) all its substitution instances is flawed, since circular. In a word, if we don't want to concede that the whole edifice of mathematics lays on loose ground, we must admit that the intuitive understanding of the conditions for proving impredicative propositions exemplified by the regress argument above is not correct.

The intuitive understanding can be expressed by the following thesis: to prove a statement of the form $\forall XA$ is to prove that $A[C/X]$ holds, where $C$ varies over all admissible substitution instances of $X$. Instead, Carnap sketches the view that proving $\forall XA$ amounts to something like proving $A[X]$ on the assumption that $X$ is an arbitrary, or generic, proposition. This understanding is explicitly attacked by Carnap:

\begin{quote}\small
If we had to examine every single property, an unbreakable circle would indeed result, for then we would run headlong against the property ``inductive''. Establishing whether something had it would then be impossible in principle, and the concept would be meaningless. But the verification of a universal logical or mathematical sentence does not consist in running through a series of individual cases [...] The belief that we must run through all individual cases rests on a confusion of ``numerical generality'' [...] We do not establish specific generalities by running through individual cases but by logically deriving certain properties from certain others. [...]

If we reject the belief that it is necessary to run through individual cases and rather make it clear to ourselves that the complete verification of a statement means nothing more than its logical validity for an arbitrary property, we will come to the conclusion that impredicative definitions are logically admissible. \cite{Carnap1983}
\end{quote}

Also G\"odel comments upon this way out of the problem:

\begin{quote}\small
[...] one may, on good grounds, deny that reference to a totality implies reference to all single elements of it or, in other words, that ``all'' means the same as an infinite logical conjunction. One may [...] interpret ``all'' as meaning analyticity or necessity or demonstrability. There are difficulties in this view; but there are no doubts that in this way the circularity of impredicative definitions disappears. \cite{Godel1944}
\end{quote}

To sum up, the epistemic problem of impredicativity revolves around an apparent circularity in the definition of the truth or provability conditions for propositions quantifying over all propositions: if one accepts the intuition that the conditions for assessing the truth or proving a statement of the form $\forall XA$ are given in terms of the conditions for assessing the truth or proving all possible substitution instances $A[B/X]$, then apparently no non-circular definition of such conditions can be given, as the proposition $\forall XA$ appears among the possible instances of $X$.

As G\"odel comments, saying that the conditions for proving (or assessing the truth) of $\forall XA$ obtain when the conditions for proving (or assessing the truth of) all the substitution instances obtain is the same as saying that the conditions for proving (or assessing the truth) of $\forall XA$ are provided by an infinite conjunction of conditions. Hence, the problem can be restated as follows: as soon as the explanation of the second order quantifier is given in terms of an infinite conjunction, that is, as soon as one accepts that a rule of the form

$$
\AXC{$\dots$}
\AXC{$A[B/X]$}
\AXC{$\dots$}
\RL{$B$ proposition}
\TrinaryInfC{$\forall XA$}
\DP
$$
provides a good explanation of the meaning of the second order quantifier, as $\forall XA$ (or a more complex statement) might occur among its immediate premisses, a problem of circularity \emph{in the definition} appears.

\subsection{Consistency proofs and the harmony between the finite and the infinite.} Hilbert's view of mathematics was based on a distinction between \emph{finitary} and \emph{infinitary} propositions: those which can be expressed in an (essentially quantifier-free) arithmetical language and those which require a more complex language, containing quantifiers ranging over infinite domains (natural numbers, real numbers, sets, \dots). In Hilbert's perspective, the first are directly meaningful and ``safe'' from a foundational perspective, as their truth can be assessed by a finite computation. On the contrary, as no direct computation can be employed to verify a proposition quantifying over an infinite domain, the use of the second must be vindicated by means of a consistency proof, showing in particular that any finitary proposition established using infinitary concepts can be also established by purely finitary means. In a word, a consistency proof should prove that any apparently redundant use of an infinitary concept can be ``eliminated''\footnote{That consistency follows from this elimination result can be seen as follows: if a proof of the absurdity $0 = 1$ exists, then a finitary proof of $0 = 1$ must also exist, which is obviously false.
}.

It is widely acknowledged that G\"odel's second theorem implies a refutation of Hilbert's reductive goal, in the sense that a reductive proof of consistency for a system containing infinitary propositions cannot be carried over within finitary mathematics. More precisely, G\"odel's second theorem does not refute the possibility of showing that redundant uses of infinitary concepts can be eliminated, but shows that a proof of this fact must
itself employ infinitary concepts\footnote{A typical example is provided by Gentzen's consistency proof of arithmetics by means of a finitary system (primitive recursive arithmetics $P RA$) augmented with the infinitary principle of transfinite induction over a recursive well-order of order type $\epsilon_{0}$.
}. In a sense, the conclusion that no false finitary con-sequence can be drawn from infinitary, unsafe, propositions, must itself be drawn from some infinitary, unsafe, proposition.

A further elaboration of Hilbert's reductive ideal appears in later proof theory, when, after Gentzen's work, the eliminability of redundant steps (called cuts or detours) in a proof became the fundamental goal of the proof-theoretic analysis of logical systems. Gentzen had shown that, in propositional (and first-order) logic, one can devise a systematic strategy to eliminate, in a proof, all redundancies which appear when some consequence is drawn from a proposition whose principal connective has just been introduced. This result, called \emph{Hauptsatz}, is the cornerstone of the proof-theoretic explanation of deduction. It has, as its main consequence, the consistency of the logical system, since, on the one hand, it asserts that any proof can be transformed into a detour-free (or cut-free) proof and, on the other hand, a proof of a contradiction must always contain a detour\footnote{In a cut-free proof of the absurd the premisses must be a formula $A$ and its negation $\lnot A$; now, the principal connective $\lnot$ of the formula $\lnot A$ must be introduced in its cut-free proof; hence this concept, $\lnot$, is introduced and eliminated in the proof, contradicting the hypothesis that the proof has no detour.
}. Moreover, it implies that a proof of a finitary proposition can be transformed into one which contains finitary propositions only (as any introduction of an infinitary proposition would be redundant and, thus, eliminable)\footnote{This is a consequence of the \emph{subformula property} of (first-order) cut-free proofs: this property asserts that all propositions occurring in the proof are subformulas of the conclusion.
}.

It is often advocated (see \cite{Prawitz1971, Dummett1991}) that the result that the \emph{Hauptsatz} obtains for a given logical system counts as a justification of the rules of that system. For instance, the philosopher Michael Dummett argues that at the basis of a proof-theoretical justification of deduction is the fact that the conditions for proving logical statements should be self-explanatory (``st\'erile'', to quote Poincar\'e), in a precise sense provided by the notion of \emph{harmony}:

\begin{quote}\small
The requirement that this criterion for harmony be satisfied conforms to our fundamental conception of what deductive inference accomplishes. An argument or proof convinces us because we construe it as showing that, given that the premisses hold good according to our ordinary criteria, the conclusion must also hold according to the criteria we already have for its holding. \cite{Dummett1991}
\end{quote}

Following Gentzen's notorious remark:

\begin{quote}\small
The introductions represent, as it were, the ``definitions" of the symbols concerned, and the eliminations are no more, in the final analysis, than the consequences of these definitions. This fact may be expressed as follows: in eliminating a symbol, the formula, whose terminal symbol we are dealing with, may be used only ``in the sense afforded it by the introduction of that symbol.'' \cite{Gentzen64}
\end{quote}

Dummett takes the introduction rule for a connective as a \emph{definition} of the meaning of the connective: the conclusion of the rule is the \emph{definiendum}, the premisses of the rule constitute its \emph{definiens}. Thus, if some consequences are drawn from a proposition A that we have previously introduced (according to the introduction rule of its principal connective) then it must be possible, by harmony, to draw the same consequences from
the propositions appearing as premisses in the rule introducing $A$; in a word, it must be possible to eliminate the connective just introduced in the derivation of its consequences by replacing it by its definiens, i.e. the premiss of the introduction rule. The justification of a logical system by means of a reduction argument, along with the thesis that introduction rules play the role of definitions of the logical constants, can be compared with Poincar\'e's idea that deductive argument must be reducible to tautological ones.

Since Gentzen's original argument, the arguments to prove the \emph{Hauptsatz} usually pro-ceed by showing how to eliminate detours gradually, i.e. by repeatedly applying a series of transformations which replace redundant concepts with their definition. For instance, if a configuration like the one below

$$
\AXC{$\stackrel{n}{A}$}
\noLine
\UIC{$\D D_{1}$}
\noLine
\UIC{$B$}
\RL{$\To$I, $(n)$}
\UIC{$A\To B$}
\AXC{$\D D_{2}$}
\noLine
\UIC{$A$}
\RL{$\To$E}
\BIC{$B$}
\DP
$$
occurs in a proof, where a consequence of the logical proposition $A\To B$ is drawn, according to its elimination rule, just after that proposition is introduced, according to its introduction rule, then it can be replaced by
$$
\AXC{$\D D_{2}$}
\noLine
\UIC{$[A]$}
\noLine
\UIC{$\D D_{1}$}
\noLine
\UIC{$B$}
\DP
$$
in which this ``detour'' has disappeared. This transformation eliminates detours only locally, since possibly some new detours are created when a certain number of copies of the derivation $\D D_{2}$ are attached to $\D D_{1}$ and the latter is attached to $\D D_{3}$. Hence, in order to complete the proof of the \emph{Hauptsatz} one must find a way to show that the iteration of this local transformations eventually eliminates all detours, globally. The crucial property in the case just illustrated is that the new detours created by the transformation occur on propositions ($A$ and $B$) which are of \emph{strictly smaller} logical complexity than the proposition $A\To B$. One can then argue by induction over the maximal logical complexity of a proposition occurring in a detour. An argument relying over an induction over an explicit discrete measure assigned to proofs is sometimes called combinatorial, for the fact that it can be formalized in first-order arithmetics.

Two are the ingredients on which a combinatorial argument for the \emph{Hauptsatz} is based: first, the fact that one can apply certain transformations to proofs which, \emph{locally}, make detours disappear (though new detours may appear elsewhere); second, the fact that the new detours created have a logical complexity strictly bounded by the logical complexity of the eliminated detour.

Gentzen's \emph{Hauptsatz} can be given an elegant description in the language of typed $\lambda$-calculi, through the lens of the propositions-as-types correspondence, that is, the acknowledgement of a substantial isomorphism between the structure of propositions and their proofs in logical systems and the structure of types and programs in typed $\lambda$-calculi\footnote{While this correspondence was initially reserved to minimal and intuitionistic logical systems, several extensions to classical systems are now well-known (see for instance \cite{parigot, Krivine09}).
}. The isomorphism between the calculus $\lambda$ and the natural deduction system $NM_{\To}$ for (the $\To$-fragment of) propositional minimal logic is a basic example of this correspondence\footnote{The isomorphism associates types with propositions by an obvious translation which simply replaces all occurrences of $\To $ by $\to$; the correspondence between typed $\lambda$-terms and derivations is obtained by associating variables with hypotheses and the abstraction and application operations, respectively, with the introduction and elimination rules for $\To $. Finally, the reduction of $\lambda$-terms is associated with the reduction of detours in natural deduction derivations.}.

Through the lens of the propositions-as-types correspondence, the redundancies in
proofs correspond to the redexes in the associated programs, i.e. to the configurations from which the execution of the program can be launched. The local transformation depicted above has a simple description in typed $\lambda$-calculus, as it corresponds to the reduction rule of $\lambda_{\to}$:

$$((\lambda x\in A)M^{B})N^{A} \ \to \ M^{B}[N^{A}/x\in A]$$
which eliminates the redex formed by an abstraction immediately followed by an application. The local elimination of redundancies corresponds then to a single step in the execution of a program, and the eliminability of all redundancies corresponds to the termination of the execution of the program. In a word, the \emph{Hauptsatz} expresses the normalization of programs, i.e. the fact that any program is associated with a (unique) normal form\footnote{As a program is said in normal form exactly when it has no redexes.
}.

The history of polymorphism is a very good example of the entanglement of proof-theory and theoretical computer science captured by the propositions-as-types correspondence: in 1971 Girard, a logician, (\cite{Girard72}) introduced System $\B F$ as an extension of the simply typed $\lambda$-calculus corresponding to intuitionistic second order logic, as a means to resolve an important conjecture in the proof-theory of second order logic. In 1974 Reynolds, a computer scientist, independently presented an extension of $\lambda_{\to}$, called polymorphic $\lambda$-calculus (\cite{Reynolds74}), in order to capture Strachey's informal notion of polymorphism by a rigorous type discipline. Actually, the logicians and the computer scientist, from different perspectives, had had the same idea: System $\B F$ and the polymorphic $\lambda$-calculus happen to be exactly the same system. This means that the representation of second order, impredicative, quantification, directly corresponds, via the propositions-as-types correspondence, to polymorphism: from a logical point of view, a polymorphic program, having all types $\sigma[\tau/X]$, with $\tau$ any type, is a program having an impredicative type constructed by second order quantification, i.e. the type $\forall X \sigma$. From a computer science point of view, a proof of an impredicative proposition $\forall XA$ is nothing but a polymorphic program. In a word, the much debated notion of second order proof is captured by the notion of polymorphic program.

In the propositions-as-types view, propositions are considered as types themselves. The proposition $\forall X(X \To X)$ is the same as the type $\forall X(X \to X)$. On the other hand, propositions need not be assigned a type: System $\B F$ does not have a type $prop$ of propositions (which would correspond to the type of a special class of types, that of
propositions). In section \ref{sec::3} we will discuss a theory in which every proposition is a type and is moreover assigned a type, the type of all types.
A precise connection exists between the apparent circularity of the provability conditions for the second order quantifier, discussed in the previous subsection, and the difficulty in showing that its rules satisfy Dummett's requirement of harmony. To in-vestigate harmony in second order logic we have to consider a configuration like the one below
$$
\AXC{$\D D_{1}$}
\noLine
\UIC{$A$}
\RL{$\forall$I}
\UIC{$\forall XA$}
\RL{$\forall $E}
\UIC{$A[B/X]$}
\DP
$$
occurring in a proof, where a consequence of the logical proposition $\forall XA$ is drawn, according to the elimination rule, just after that proposition is introduced, according to the introduction rule; a local transformation to eliminate this detour locally can be applied, yielding the configuration below
$$
\AXC{$\D D_{1}[B/X]$}
\noLine
\UIC{$A[B/X]$}
\DP
$$
where $\D D_{1}[B/X]$ is the derivation obtained by systematically replacing in $\D D_{1}$ all occurrences of $X$ with $B$. The local transformation just depicted has a very simple description in terms of typed $\lambda$-calculus, as it corresponds to the execution rule of System $\B F$:
$$
(\Lambda x.M^{A})\{B\} \ \to \ M[B/X]^{A[B/X]}
$$
which eliminates the redex formed by a type abstraction immediately followed by a type extraction. The local elimination of redundancies is then a single step in the execution of a program.

As in the propositional case, the execution of the local transformation might create new detours; however, a combinatorial bound on the logical complexity of the formula occurring in this detour can not be given in this case: as $B$ might be any formula (in particular $B$ might be the formula $\forall XA$ itself)\footnote{A consequence of this remark is that the subformula principle does not hold for second order logic: for instance, a cut-free proof of a proposition of the form $\forall XA \To B$ might contain propositions of arbitrary logical complexity.
}, the new detours might be arbitrarily complex. In a word, the fact that the same formula $\forall XA$ can occur as a possible substitution for its variable $X$ (i.e. Russell's vicious circle) blocks any combinatorial argument for proving that detours in second order logic can be eliminated.

The \emph{Hauptsatz} for second order logic was conjectured in a paper by Takeuti in 1954 and independently proved by Tait, Prawitz and Takahashi (\cite{Tait1966, Takahashi1967, Prawitz1968}) by means of semantical (i.e. non combinatorial) techniques. However, as it will be discussed in more detail in the next section, the first proof of the \emph{Hauptsatz} for second order logic based on a reduction argument (i.e. showing that the local reduction like the one above eventually eliminate all detours) is a corollary of Girard's 1971 normalization theorem for System $\B F$. Girard's non combinatorial argument exploits a series of predicates of growing logical complexity (hence not globally definable in second-order arithmetics, see section \ref{sec::3}).

His result proves that, in all second order proofs, a redundant concept can be eliminated in favor of its definition.
As a consequence of the second order \emph{Hauptsatz}, harmony does hold for the second order quantifier. However, an epistemological problem in the proof-theoretical justification of second order logic stems from the fact that the argument for the \emph{Hauptsatz} cannot be of a combinatorial nature and must rely on a heavy (and impredicative) logical apparatus. In a word, even if every redundant use of the infinite concept of second order quantification can be eliminated in a finite amount of time, no definitive reduction of the infinite to the finite is achieved, as infinitary concepts are needed to establish this fact. This is compatible with our previous remarks on G\"odel's second theorem: the justification of an infinitary system like second order logic should require infinitary concepts.

Nevertheless, it is often argued that the second order \emph{Hauptsatz} does not produce a viable justification of the use of second order concepts, as it does not provide a reduction of impredicative logic to a predicative ground:

\begin{quote}\small
[...] it should be noted, [...] that these new proofs use essentially the principles formalized in $P^{2}$ [second order Peano Arithmetics], and there is thus still no reduction of e.g. the impredicative character of second order logic. \cite{Prawitz1971}
\end{quote}

In particular, to the eyes of a sceptic with respect to impredicative arguments, an impredicative proof that harmony obtains in second order logic might not provide a satisfying justification of the use of second order quantification.

A different form of circularity, then, with respect to the one discussed in the last sub-section, appears in the proof-theoretic justification of the rules for second order quantification: this is the particular form of circularity appearing in the argument showing that harmony (or eliminability) holds between conditions for introducing and for eliminating the second order quantifier. The fact that all instances $A[B/X]$ can be consequences of $\forall XA$ makes the definition of a combinatorial measure on proofs, which should decrease along with the progressive elimination procedure, impossible. Rather, the arguments showing that such detours are eventually eliminated during the reduction process rely over principles that, in a sense, reflect the impredicativity to be justified. This form of circularity will be discussed in detail in section \ref{sec::3}.

\section{Propositions about all propositions.}\label{sec::2}
 The problem of the circularity \emph{in the definition} of the truth or provability conditions for second order logic is investigated by comparing two arguments: first, Frege's consistency argument in the \emph{Grundgesetze}, purported to show that all propositions in the system have a definite truth-value and based on a definition of what it means for an expression of a given level to have a denotation; second, Girard's normalization proof for System $\B F$ corresponding, through the propositions-as-types correspondence, to a consistency proof for second order logic and based on a definition of what it means for a term of a given type to be computable.

\subsection{Frege's consistency proof.} The language of the \emph{Grundgesetze} (let us call it $\C G$) would be called nowadays a functional language. It was based on two basic distinctions: firstly, the one between functions and objects, the first being unsaturated entities, i.e. constitutively demanding for an argument (a function or an object, depending on the level of the function - see below) to be fed with, the second being, on the contrary, saturated, i.e., let us say, complete, entities. Secondly, the one between names and
their denotations: to every class of entities (objects and functions of some type), which forms a well-defined totality, there corresponds a class of names of which those entities are the denotations. Correspondingly, names for functions, called function names, are unsaturated expressions, (one would say open terms, i.e. terms containing free variables, in modern language), whereas names for objects, called proper names, are saturated expressions (closed terms, i.e. terms without free variables).

Interestingly, the language $\C G$ contains no primitive saturated expressions. Numerical expressions $0,1,2,\dots$, for instance, are famously (or infamously) constructed by Frege starting from unsaturated expressions. A peculiar class of saturated expressions is the class of propositions, which are, in Frege's terminology, names for the most basic objects, truth-values, i.e. True or for the False.

Functions are divided into three classes: \emph{first-level functions}, designated by unsaturated expressions of the form $f(x),g(x),\dots$, whose free variables $x,y,z,\dots$ can be substituted for proper names, yielding proper names; \emph{second-level functions}, designated by unsaturated expressions of the form $\phi(X),\psi(X),\dots$, whose free variables $X,Y,Z,\dots$ can be substituted for first-level function names, yielding proper names; \emph{third-level functions}, designated by unsaturated expressions of the form $\Phi(\eta), \Psi(\eta),\dots$, whose free variables $\eta,\theta,\zeta$ can be substituted for second-level function names, yielding proper names. Examples of first-level functions are given by the logical connectives, e.g. $f(x)=x\to x$, $g(x)=\lnot x$ and by truth-functions, e.g. $x^{2}-1=(x+1)\times(x-1)$; examples of second-level functions are the first-order quantifiers $\forall x X(x), \exists xX(x)$ as well as the function $\spirituslenis{\varepsilon}X(\varepsilon)$, which allows to construct, from any first-level function $f(x)$, a \emph{course-of-value expression}, i.e. a saturated expression $\spirituslenis{\varepsilon}f(\varepsilon)$ designating the class of all objects falling under the concept expressed by the function $f(x)$. Finally, examples of third-level functions are the second-order quantifiers $\forall X\eta(X), \exists X\eta(X)$.

The paragraphs \S29-31 of the \emph{Grundgesetze} contain an argument purported to show that every expression in $\C G$ has a (unique) denotation, i.e. is the name of a well-defined object or function. First, Frege provides conditions for an expression of $\C G$ to have a denotation.

Two cases are considered: if $e$ is a saturated expression, Frege stipulates that $e$ has a denotation if the result $f(e)$ of replacing by $e$ the free variable of any denoting unsaturated expression $f(x)$ (of the appropriate type) is a denoting saturated expression. If $e(\vec x)$ is an unsaturated expression, depending on the variables $\vec x$, Frege stipulates that $e(\vec x)$ has
a denotation if the result $e(\vec b)$ of replacing its free variables with denoting expressions $\vec b$ of the appropriate type always yields a denoting saturated expression. Here one might object that Frege's stipulations are circular, as the denotation of saturated expressions and unsaturated expressions depend on each other. Actually, the German logician denies that his stipulations form a definition of the notion of ``having a denotation''. He states that his stipulations

\begin{quote}\small
[...] are not to be construed as definitions of the words ``to have a reference'' or ``to refer to something'', because their application always assumes that some names have already been recognized as having a reference; they can however serve to widen, step by step, the circle of names so recognized. \cite{Frege1893}
\end{quote}

This objection is discussed for instance by Dummett (\cite{Dummett1991b}, p. 215). We will not enter this delicate aspect of Frege's argument: we will concede that Frege is right here.

The second part of Frege's argument is aimed at showing that any proper name has a denotation, hence in particular that propositions have a definite truth-value. Frege argues by induction on the construction of a proper name, by applying the stipulations at each step. As there is no primitive saturated expression, proper names (as the propositions $5 + 5 = 2 \cdot  5 $ or $\forall x(x + x = 2 \cdot x)$, $\forall X(X \to X)$) must be construed by filling the hole in a first, second or third-level function name ($x + x = 2\cdot x$, $\forall xX(x)$, $\forall X \eta(X)$, respectively) with an expression of the appropriate type (the name $5$, the first-level function $x + x = 2 \cdot x$, the second-level function $X \to X$, respectively). Hence the argument is reduced to showing that primitive unsaturated expressions, of any level, have a denotation.
Remark that, since any proposition is built by opportunely composing function names, the conditions under which a proposition denotes the True are univocally determined once the conditions of what it means to denote for an unsaturated expression are determined. Hence, if we exclude the case of course-of-values expressions, Frege's second stipulation should suffiƒce to reconstruct, from the inductive argument, a unique truth-value for propositions.

For first-level function names $f(x)$ one must show that, if $N$ is a denoting name, then $f(N)$ is too; for instance, if $f(x)$ is the function name $x \to \lnot x$, and $P$ is, by hypothesis, a denoting proposition (corresponding to a definite truth-value), then $P \to \lnot P$ is a denoting proposition (as its truth-value is univocally determined by the truth-value of $P$).

Frege argues in a similar way for second-level expressions, like the function $\forall xX(x)$. On the assumption that $f(x)$ is a denoting first-level function one must show that the proposition $\forall xf(x)$ denotes either the True or the False; if $f(x)$ is denoting, then, for any denoting proper name $N$, the proposition $f(N)$ denotes either the True or the False; now, either $f(N)$ denotes the True for all denoting $N$ or, for some denoting $N$, $f(N)$ denotes the False. In the first case, then $\forall xf(x)$ denotes to the True; in the second case, it denotes to the False; hence, in any case, the proposition $\forall xf(x)$ denotes either the True or the False.

An apparently circular dependency can be found in this argument, though: the verification that the unsaturated expression $\forall xX(x)$ denotes depends on the verification that the saturated expression $\forall xf(x)$ denotes for every choice of a denoting first-level function name $f(x)$, hence including names built using the unsaturated expression $\forall xX(x)$.

This circularity can more clearly be detected in the case of third-level expressions. Frege simply states that in this case one can argue similarly to the second-level case. In order to show that the third-level function $\forall X\eta (X)$, where $\eta$ is a variable which stands for the name of a second-level expression, denotes, one must show, on the assumption that $\Phi(X)$ is a denoting second-level expression, that the proposition $\forall X\Phi (X)$ denotes either the True or the False; now, either $\Phi(f)$ denotes the True for all denoting first-level function name $f$, or for some denoting $f$, $\Phi(f)$ denotes the False. In the first case, then, $\forall X \Phi(X)$ denotes the True; in the second case it denotes the False; hence, in any case, the proposition $\forall X \Phi(X)$ denotes either the True or the False.

As remarked above, this argument should grant that a proposition built by replacing the variable $\eta$ with a denoting second-level expression $\Phi(X)$ in the third-level expression $\forall X \eta(X)$ denotes a definite truth-value. Let $\Phi(X)$ be the second-level function name $\forall X(X(x)\to \lnot X(x))$; this function name is denoting as we have shown that the second-
level function $\forall xX(x)$ is denoting and that, if $f(x)$ is a denoting first-level function,
then $f(x) \to \lnot f(x)$
is a denoting first-level function.  Now the name $ \forall X \Phi(X)$, that
is, $\forall X\forall x (X(x)\to \lnot X(x))$
denotes the True if and only if, for every denoting first-level
expression $f(x)$, the name $\forall x(f(x) \to \lnot f(x))$ denotes the True; hence, in particular, only if the name $\forall x(g(x)\to \lnot g(x))$ denotes the True, where $g(x)$ is the denoting expression $\forall X\Phi(X) \to \lnot (x=2\cdot x)$, which is the case if and only if $\forall X \Phi (X)$ denotes the True.

Similarly to
Carnap's
example
in the previous section, in order to know whether
$\forall X\Phi (X)$ denotes the True, we must verify whether all of its substitution instances $\Phi (f)$, where $f$ is a denoting first-level function name, denote the True. Now, if $f$ is built from $\forall X \Phi(X)$, we are led into a troublesome epistemic position: we must already know whether $\forall X\Phi (X)$ denotes the True in order to assess whether $\forall X\Phi (X)$ denotes the True.

A similar form of circularity appears in the case of course-of-values expressions. Frege
shows that the second-level function $\spirituslenis{\varepsilon}X(\varepsilon)$ has a denotation by arguing that, for any two
denoting first-level function names $f(x), g(x)$, the expression $g(\spirituslenis{\varepsilon}f(\varepsilon))$ has a denotation (exploiting his first stipulation). For that he considers the possible cases for $g(x)$, taking as basis case the one of equality and appealing to the celebrated and unfortunate Basic
Law V (stating that two course-of-value expressions $\spirituslenis{\varepsilon}f(\varepsilon),\spirituslenis{\varepsilon}g(\varepsilon)$ denote the same object
if and only if the proposition $\forall x(f(x) \leftrightarrow g(x))$ denotes the True).

A vicious circle arises when we try to compute the denotation of an arbitrary course-
of-value expression $\spirituslenis{\varepsilon}h(\varepsilon)$: by the first stipulation, we must show that, for any denoting
first-level function name $g(x)$, $g(\spirituslenis{\varepsilon}g(\varepsilon)$
has a denotation.
Let $g(x)$
be the function
$x=\spirituslenis{\varepsilon}h(\varepsilon)$; in order to show that the
proposition $g(\spirituslenis{\varepsilon}h(\varepsilon))$, that is, $\spirituslenis{\varepsilon}g(\varepsilon)=\spirituslenis{\varepsilon}h(\varepsilon)$ has
a denotation, we must show that the proposition $\forall x(g(x)\leftrightarrow h(x))$ has a denotation.
This means that we must show that, for every denoting proper name $N$, $g(N) \leftrightarrow h(N)$
denotes either the True or the False. In particular, then, we are led to show that $g(\spirituslenis{\varepsilon}h(\varepsilon))$
denotes the True, i.e. that $\spirituslenis{\varepsilon}g(\varepsilon)=\spirituslenis{\varepsilon}h(\varepsilon)$ has a denotation, giving rise to a regress.

It is widely accepted that the main problem of Frege's system lies in the presence of course-of-value expressions (which reproduce a sort of na\"ive set-theory inside $\C G$), and which allow to construct the Russell sentence; in particular, Heck (\cite{Heck1998}) proved that the system $ \C G$ without course-of-values is consistent. Nevertheless, from the remarks above it follows that, even without such constructions, Frege's consistency argument would be fallacious: as it is stipulated, the notion of denotation is circularly stipulated (this view is defended by Dummett in his analysis of Frege's argument, \cite{Dummett1991b}). Indeed, the denotation for unsaturated expressions is determined directly in terms of the denotation of a class of expressions constituting possible substitution instances for their free variables, containing just those unsaturated expressions for which the notion of denotation is being stipulated.

\subsection{Girard's consistency proof.} Since Frege's failure, several proofs of consistency for second order logic (without course-of-values expressions) have been provided.

The standard model-theoretic interpretation of second order logic appeals to an assignment of elements of a model to the free variables. More precisely, if $A(\vec X, \vec x)$ is a formula,
where $\vec X$ denotes its predicate variables (each one with a fixed arity) and $\vec x$ its individual variables, the interpretation $A_{\C M}[s]$ of $A$ over a model $\C M$ (of support $M$) depends on a map $s$ associating subsets of $M^{n}$ to $n$-ary predicate variables and elements of $M$ to the individual variables. The model $\C M$ satisfies the proposition $\forall XA$ \emph{relatively to the
assignment $s$}, when, for all $e\subseteq M^{n}$ (where $n$ is the arity of $X$), $\C M$ satisfies $A$ relatively to the assignment $s_{X\mapsto e}$, where $s_{X\mapsto e}$ is the assignment which differs from $s$ only in that it assigns $e$ to the variable $X$.

Contrarily to Frege's stipulation of denotations, the satisfaction property depends on the choice of an assignment of values (denotations, one might say) to variables. Contrarily to Frege's property of denotation, the property ``the model $\C M$ satisfies proposition $A$ relative to an assignment $s$'' is correctly defined by induction over the complexity of formulas: the verification that $\C M$ satisfies $\forall XA$ depends on the verification that $\C M$ satisfies the simpler proposition $A$ in a (possibly uncountable) number of different cases, depending on all assignments $s_{X\mapsto e}$, for each subset $e\subseteq M^{n}$. If we wish to express the property as an infinite conjunction, we can use the infinitary (uncountable, in general) rule below:
$$
\AXC{$\dots$}
\AXC{$\C M\vDash A[s_{X\mapsto e}]$}
\AXC{$\dots$}
\RL{$e\subseteq M^{n}$}
\TrinaryInfC{$M\vDash \forall XA[s]$}
\DP
$$

Following the inductive definition of the satisfaction relation, one can correctly show that all formulas in second order logic have, under a given assignment, a definite denotation, without falling into the regress of Frege's argument. Hence, every closed formula receives a unique truth-value under a given interpretation.

We can compare the definition of denotation given in terms of a model-theoretic satis-faction relation with Frege's stipulations of the property of having a denotations. Frege does not interpret expressions by elements of an arbitrarily chosen model; in some sense, the language $\C G$ comes with an intended interpretation. Variables of an appropriate type are intended to vary over the domain made of the ``totality'' of the objects or functions of that type. Unsaturated expressions designate operations which can be applied on the domains associated with their free variables (for instance, the unsaturated expression $x \to x$ designates a well-defined operation over truth-values).

While the model-theoretic notion of denotation is relative to an assignment of elements of the model to the free variables, Frege's notion of denotation is absolute and rests on the idea of the existence of a totality of the objects or functions of a given level: the denotation of an unsaturated expression depends on the denotation of all expressions obtained by closing the former with saturated expressions. Hence, the denotation of an expression $e( x)$ depends on all substitution instances $e( b)$, where $ b$ varies among names for the objects or functions belonging to the appropriate ``totality'' (a more detailed comparison, as well as the proposal of a different reading of Frege's argument, is provided by Heck \cite{Heck1998}).

Even if circularity no more occurs in the verification of the truth-value of a formula under an interpretation, a different epistemic difficulty arises, if provability conditions are considered rather than truth-conditions: one would hardly accept that, in order to prove an (interpreted) universally quantified statement $\forall XA$ one has to prove all possible instances $A[s_{X\mapsto e}]$,where $e$ varies among the (possibly uncountably many) subsets of the domain of interpretation. In a word, verification is no more circular, but it is wildly infinitary.

A second epistemic objection is that the semantical explanation does not seem in accordance with Dummett's requirement that the consequences of a logical proposition
should be derivable already from the conditions for proving the proposition: among the immediate consequences of $\forall XA$ we must count all the substitution instances $A[B/X]$, where $B$ is an arbitrary proposition; however, such instances are not among the conditions for deriving $\forall XA$, as these are all of the form $A[s_{X\mapsto e}]$, where $e$ is a set. In a word, \emph{when one would like formulas, one actually finds sets} (their possible interpretations).

A consistency argument obtained by model-theoretic means will then hardly satisfy the demands of the logician who sincerely doubts that justifiable provability conditions can be provided for second order logic. Rather, following the arguments sketched in the previous section, he will demand that consistency be derived as a consequence of the \emph{Hauptsatz}, i.e. the result showing that all detours in a proof can be eliminated.

As we mentioned, the \emph{Hauptsatz} for second order logic is a direct consequence of the normalization theorem for System $\B F$. This is actually the reason why this system was originally created. In the brief reconstruction that follows, we will highlight how Girard's normalization argument provides a sophisticated way to escape the problematic circularity in Frege's argument.

Girard's proof is based on the generalization of a technique developed by Tait to prove normalization for simple types. In 1967 Tait (\cite{Tait1966}) had invented a non-combinatorial technique to prove normalization for typed $\lambda$-calculi, and had applied it successfully to $\lambda_{\to}$ and to the more complex type system $\B T$, first introduced by G\"odel, and essentially corresponding to first-order Peano Arithmetics. Tait's idea was to associate, with each type $\sigma$, a special predicate $Comp_{\sigma}$ over $\lambda$-terms, the \emph{computability of $\sigma$}, which intuitively expresses the fact that the term is normalizable (in particular, if $M$ is computable, then it has a normal form). Computability predicates are defined by induction over types. The most interesting case is that of a type of the form $\sigma\to \tau$ : the predicate $Comp_{\sigma\to \tau}$ is said to hold for a term $M$ if, for every term $N$ for which the predicate $Comp_{\sigma}$ holds, the predicate $Comp_{\tau}$ holds for the term $M N$ (i.e. the application of $M$ to $N$). Hence, a term is computable for a type $\sigma\to \tau$ if, whenever applied to a term computable in $\sigma$, it yields a term computable in $\tau$ . As we anticipated, this is an absolute definition, not depending on any assignment.

Normalization is proved (see \cite{Tait1966}) by showing, by induction on a derivation of a typing judgement of the form $\Gamma\vdash M\in \sigma$ (where $\Gamma$ denotes a finite context made of typing assumptions of the form $(X_{i}\in \sigma_{i})$ for the free variables $x_{i}$ occurring in $M$), that for every choice of computable terms $N_{1},\dots, N_{p}$ for the types occurring in $\Gamma$, the term $P=M[N_{1}/x_{1},\dots, N_{p}/x_{p}]$ is computable of type $\sigma$  (i.e. that $Comp_{\sigma} (P)$ holds).

The technique of computability predicates is non combinatorial because it exploits a sequence of predicates of growing logical complexity, making the whole proof not formalizable in first-order Peano Arithmetics $\B{PA}$. This is in accordance with G\"odel's second theorem, since with this technique Tait was able to prove the normalization of system $\B T$, a result implying the consistency of $\B{PA}$.

Girard's original idea was to extend G\"odel's \emph{Dialectica} interpretation of $\B{PA}$ into System $\B T$ to second order Peano Arithmetics $\B{PA}^{2}$. This is the reason why he introduced System $\B F$. In order to prove the normalization of this new system, he generalized Tait's technique to the types of System $\B F$. By G\"odel's second theorem, as the normalization of System $\B F$ implies the consistency of $\B{PA}^2$, such an argument should not be formalizable
in the latter system. Hence, it was expectable that the correct definition of computability for System $\B F$ types be not formalizable in the language of $\B{PA}^2$.

Like Frege's, Tait's definition of computability (\cite{Tait1967}) is absolute: it does not depend on an assignment nor a model, but is given directly in terms of the substitutions instances obtained by replacing computable terms for the free variables occurring in a term. The extension of Tait's technique to impredicative types, in an absolute way, produces circularities similar to those in Frege's stipulations: the computability predicate for an impredicative type $\forall X \sigma$ should depend upon \emph{all} computability predicates $Comp_{\sigma}$ , for all type $\sigma$, hence, in particular, it should depend on itself.

\begin{quote}\small
We would like to say that $t$ of type $\forall X.\tau$ is \emph{reducible}\footnote{Actually, Girard calls ``reducibility'' the property Tait calls ``computability''.
} if for all types $\sigma$, $t\sigma$  is reducible (of type $\tau[\sigma/X]$) [...] but $\sigma$ is arbitrary - it might be $\forall X.\tau$ - and we need to know the meaning of reducibility of type $\sigma$ before we can define it! We shall never get anywhere like this. (GLT 1989)
\end{quote}

In other words, if one defined the computability for the type $\forall X\sigma$ in an absolute way, i.e. by stating that $Comp_{\forall X\sigma}$ holds of $M$ if, for all type $\tau$, $Comp_{\sigma[\tau/X]}$ holds of $M\{\tau\}$, one would reproduce the circularity in Frege's stipulations. Expressed as an infinite conjunction, the predicate $Comp_{\forall X\sigma}$ would be defined by the inference rule below
$$
\AXC{$\dots$}
\AXC{$Comp_{\sigma[\tau/X](M\{\tau\})}$}
\AXC{$\dots$}
\RL{$\tau$ type}
\TrinaryInfC{$Comp_{\forall X\sigma}(M)$}
\DP
$$
which is not well-founded. For instance, with $\sigma = X$, the rule gives rise to an infinite path:
$$
\AXC{$\vdots$}
\noLine
\UIC{$Comp_{\forall XX}(M\{\forall X X\}\{\forall X X\})$}
\RL{$\forall X X$ type}
\UIC{$Comp_{\forall XX}(M\{\forall X X\})$}
\RL{$\forall XX$ type}
\UIC{$Comp_{\forall X X}(M)$}
\DP
$$

As previously discussed, satisfaction in model-theoretic semantics is defined relatively to an assignment of sets to variables. Hence the quantifier $\forall$ is interpreted as ranging not over formulas but over those \emph{sets which can be the interpretation of a formula}. The passage to a relative notion of computability constitutes then one of the keys of the normalization argument for System $\B F$. Introducing a relative notion of computability means to interpret the quantifier $\forall$ not as ranging over types but rather as ranging over those \emph{sets which can be the interpretation of a type}, i.e. over computability predicates. Girard's \emph{tour-de-force} technique was to introduce a general notion of computability predicate of which the predicates of computability for System $\B F$ types are particular instances. For that he introduced \emph{reducibility candidates}, which are sets of normalizable
$\lambda$-terms closed with respect to certain properties related to the execution laws for $\lambda$-terms\footnote{Several definitions of reducibility candidates exist in the literature. For a comprehensive discussion, see \cite{Gallier}.
}.

The computability predicates $Comp_{\sigma}$ are defined relatively to an assignment $s$ of reducibility candidates to the type variables. If $\sigma$ is a variable $X$, then $Comp_{X}$ holds of $M$
\emph{relative to $s$} when $M$ belongs to the candidate $s(X)$. If $\sigma$ is of the form $\tau\to \rho$, then computability is obtained by appropriately modifying Tait's stipulation: the predicate
$Comp_{\sigma}$ holds of $M$ \emph{relative to $s$} if, for all term $N$ for which $Comp_{\sigma}$ holds \emph{relative to $s$}, $Comp_{\tau}$ holds of $MN$ \emph{relative to $s$}.
Finally, in	is a universally quantified type $\forall X\sigma$ , the predicate $Comp_{\forall X\sigma}$	holds of $M$
\emph{relative to $s$} if, for all type $\tau$ and for all reducibility candidate $\C E$ for the type $\tau$, the predicate $Comp_{\sigma}$ holds of $M\{\tau\}$ relative to $s_{X\mapsto \C E}$, where $s_{X\mapsto \C E}$ the assignment which differs from $s$ only in that it associates the candidate $\C E$ to the variable $X$. This stipulation can be expressed by the infinitary rule below:
$$
\AXC{$\dots$}
\AXC{$Comp_{\sigma}[s_{X\mapsto \C E}](M\{\tau\})$}
\AXC{$\dots$}
\RL{$\tau$ type, $\C E\in \C{CR}_{\tau}$}
\TrinaryInfC{$Comp_{\forall X\sigma}[s](M)$}
\DP
$$
where $\C{CR}_{\tau}$ indicates the set of reducibility candidates for the type $\tau$. The definition is well-founded, as the property of relative computability for the type $\forall X\sigma $ is defined only in terms of (possibly uncountably many) cases of relative computability for the strictly simpler type $\sigma$.

Generalizing Tait's argument, Girard finally proved, by induction on the derivation of a typing judgement of the form $M\in \sigma$, that the term $M$ is computable of type $\sigma$ with respect to any assignment $s$ \footnote{Here, again, due to the presence of free variables, the statement is actually more complex: if $M\in \sigma$
then, for every assignment $s$ and every choice of terms $\vec N$ which are computable (in the appropriate
types) relative to $s$, the term $M[\vec N/\vec x]$ is computable for	$\sigma$ relative to $s$.
}. The most relevant case of the inductive argument is that of a derivation ending with the rule corresponding to type extraction:
$$
\AXC{$M\in \forall X\sigma$}
\RL{$Ext$}
\UIC{$M\{\tau\}\in \sigma[\tau/X]$}
\DP
$$

If M is computable for the type $\forall X\sigma$ relative to $s$, then, for all type $\tau$, one should be able to derive that $M\{\tau\}$ is computable of type $\sigma[\tau/X]$ relative to $s$; however, among the conditions for deriving that $M$ is computable for the type $\forall X\sigma$ relative to $s$, there is only the fact that $M\{\tau\}$ is computable of type relative to all $s_{X\mapsto \C E}$. Similarly to the model-theoretic case, \emph{where one would like a type, one finds a set}. One might then suspect that, similarly to the model-theoretic explanation, the computability explanation is not in accordance with Dummett's requirement that the consequences of a logical proposition should be derivable already from the conditions from proving the proposition.

Nevertheless, in the normalization argument this problem is solved by means of a lemma, usually referred to as the substitution lemma (elsewhere as Girard's trick, \cite{Gallier}), which, from a foundational viewpoint, constitutes probably the most interesting part of the proof (and will be discussed in detail in the next section):

\begin{lemma}[substitution lemma]
For any type $\sigma$ the computability of type $\sigma$ relative to the assignment $s_{X\mapsto Comp_{\tau}[s]}$ is equivalent to the computability at type $\sigma[\tau/X]$ relative to $s$.
\end{lemma}

If $M$ is computable for the type $\forall X\sigma$ relative to $s$, then for all type $\tau$, $M\{\tau\}$ is computable for $\sigma$ relative to all $s_{X\mapsto \C E}$ (where $\C E$ is any reducibility candidate for the type
$\tau$). Hence in particular $M\{\tau\}$ is computable of type relative to $s_{X\mapsto Comp_{\tau}[s]}$. Now, by the substitution lemma, this implies that $M\{\tau\}$ is computable of type $\sigma[\tau/X]$ relative to $s$. In other words, we can argue, without falling into regress, that, if $M$ is computable for $\forall X\sigma$ relative to $s$ then, for all type $\tau$, $M\{\tau\}$ is computable of type $\sigma[\tau/X]$ relative to $s$. Hence the fact that, if $M$ is computable of type $\forall X\sigma$ relative to $s$, then $M$ is computable
of type $\sigma[\tau/X]$ relative to $s$, holds by proof (a proof appealing to a quite delicate lemma, as we discuss in section \ref{sec::3}) and not hold by the very definition of computability.

The appeal to a relative notion of computability (obtained thanks to the notion of reducibility candidates) and to the substitution lemma allows to overcome the apparent circularity of the provability conditions for impredicative propositions: a ``computable'' proof for $\forall XA$ yields ``computable'' proofs for all possible substitution instances $A[B/X]$, included those from which one would expect a circularity to appear (i.e. $B = \forall XA$).

A first glance at Girard's argument might give the impression that the substitution lemma makes circularity disappear as by a magician's trick. In the next section, starting with Martin-L\"of's unfortunate attempt to generalize this argument, we will try to uncover the gimmick, and to reconstruct the sophisticated form of circularity that the argument exploits.

\subsection{Computability and uniform proofs.} Though with a quite heavy conceptual baggage, Girard's argument shows that there exists a procedure which eliminates, in a finite amount of time, all detours from a second order proof. This argument, as we saw, seems to vindicate the idea that the consequences of a universally quantified proposition are already implicit in the conditions for proving such a propositions. However, as in the case of the model-theoretic explanation, it seems implausible to identify such con-ditions with a (possibly non countable) family of sub-derivations, as the definition of computability seems to suggest.

Nevertheless, the computability argument does shed some light over the problem of identifying the provability conditions for impredicative propositions. If one considers the infinitely many conditions for asserting that a $\lambda$-term $M$ is computable of type $\forall X\sigma$, one realizes that there is just one thing which remains constant in all those conditions: this is the term $M$ itself. Indeed, for any type $\tau$ and for any reducibility candidate $\C E$ for $\tau$, it is the same $M$ for which computability of type $\sigma$ relative to $s_{X\mapsto \C E}$ holds\footnote{One might object that it is not $M$ but rather $M\{\tau\}$ which is computable at type $\sigma$ relative to $s_{X\mapsto \C E}$ and hence that the term whose computability is predicated is not constant in all premises. However, one can rewrite the normalization proof with the version \emph{\`a la Curry} of System $\B F$ (see footnote \ref{footnote::8}), by exploiting an untyped variant of reducibility candidates (see for instance \cite{Gallier}), where these are sets of untyped $\lambda$-terms and do not depend on types. In this setting the computability condition can be written in the form
$$
\AXC{$\dots$}
\AXC{$Comp_{\sigma}[s_{X\mapsto \C E}](M)$}
\AXC{$\dots$}
\RL{$\C E\in \C{CR}$}
\TrinaryInfC{$Comp_{\forall X\sigma}[s](M)$}
\DP
$$
where it is clear that it is the same $M$ which is uniformly computable at type $\sigma$ for all assignment $s_{X\mapsto \C E}$.
}. This looks very similar to what happens with the program $\TT{Map}$ described in the previous section: the program is not intended to provide, for any given type, a specific behavior; rather, its behavior is described so as to be the same for every type. In a sense, the fact that a program is computable at type $\forall X\sigma$ means that that very program will be computable at type $\sigma$, for any choice of a type $\tau$ and of an appropriate candidate, and hence, eventually (by the substitution lemma), computable at type $\sigma[\tau/X]$.

With Girard's own words,

\begin{quote}\small
In other words, everything works as if the rule of universal abstraction (which forms functions defined for every type) were so \emph{uniform} that it operates without any information at all about its arguments. \cite{Girard1989}
\end{quote}

The uniformity of computability predicates constitutes, historically, one of the first mathematical formalizations of the idea of parametric polymorphism introduced by Strachey in \cite{Strachey1967} (see above). The uniformity of second order quantification is highlighted by a quite surprising result, sketched in Girard's thesis: as a consequence of the normalization theorem, no program $M$ in System $\B F$ can be non-uniformly polymorphic in the sense that, for two distinct types $\sigma,\tau$ , the terms $M\{\sigma\}$ and $M\{\tau\}$ reduce to distinct normal forms. This means, intuitively, that polymorphic programs in System $\B F$ cannot behave differently depending on the type which is assigned to them, coherently with the idea that a polymorphic program is insensible to the type assigned. A very simple example of a non-uniform polymorphic program would be, for instance, a program $\TT P$ of type $\forall X\forall Y(X\to Y\to \TT{Bool})$\footnote{Where $\TT{Bool}$ is the type $\forall Z(Z\to Z\to Z)$ containing the codes $\TT T = \Lambda Z.(\lambda x\in Z)(\lambda y\in Z)x$ and $\TT F=\Lambda Z.(\lambda x\in Z)(\lambda y\in Z)y$ for the truth-values.
} deciding whether two closed types are equal, i.e. satisfying the instructions below:
$$
\TT P\{\sigma\}\{\tau\} =
\begin{cases}
\TT T & \text{ if }\sigma,\tau \text{ are closed and }\sigma=\tau \\
\TT F & \text{ otherwise}
\end{cases}
$$

We recall here a variant of Girard's original argument (contained in \cite{Mitchell1999}): we suppose the existence of a non-uniformly polymorphic term, and we use it to construct a counterexample to the normalization theorem. Let us suppose that System $\B F$ contains a term $\TT J$ of type $\forall X\forall Y((X\to X)\to (Y\to Y))$ whose execution depends upon the types instantiated for $X$ and $Y$ in the following sense: if the same closed type $\sigma$ is instantiated for $X$ and $Y$, and if $M$ is a term of type $\sigma\to \sigma$, then $\TT J$, applied to $M$, gives $M$ back; on the contrary, if two distinct closed types and are instantiated for $X$ and $Y$, respectively, and if $M$ is a term of type $\sigma\to \sigma$, then $\TT J$ applied to $M$ erases $M$ and outputs the term $\TT{ID}_{\tau}$. The program $\TT J$ satisfies then the instructions below:
$$
\TT J\{\sigma\}\{\tau\} M =
\begin{cases}
M & \text{ if }\sigma,\tau \text{ are closed and }\sigma=\tau \\
\TT{ID}_{\tau} & \text{ otherwise}
\end{cases}
$$

The polymorphism of the term $\TT J$ is then non-uniform, since its output depends, similarly to $\TT P$, upon a former verification of the equality between the types which are assigned to it.

We show now how, on the basis of our assumption, a term of type $\rho= \forall Z(Z\to Z)$ having no normal form can be constructed, contradicting the normalization theorem: let $\Delta$ be the term $(\lambda x\in \rho)(x\{\rho\})x$, of type $\rho\to \rho$. Then the term $\TT K := \Lambda z.(\TT J\{\rho\}\{Z\})\Delta$ has type $\forall Z(Z\to Z)$. In definitive the term $(\TT K\{\rho\})\TT K$ has type $\rho$ and reduces in a finite number of steps to itself:
$$
(\TT K\{\rho\})\TT K \ \to \ (\TT J\{\rho\}\{\rho\}\Delta)\TT K \ \to \ \Delta \TT K \ \to \ (\TT K\{\rho\})\TT K
$$
hence it cannot be normalizable.

From the logical viewpoint, the argument can be summarized as follows: as soon as one admits that, in order to devise a proof of a universally quantified proposition $\forall XA$, one can choose a different argument for different possible substitution instances for $X$, then one can actually exploit the circularity of second order types to construct a proof whose redundancies cannot be eliminated. Hence, the \emph{Hauptsatz} of second order logic
implies, in addition to consistency, that the proofs of an impredicative proposition cannot be non-uniform.

These considerations have a quite strong consequence regarding the first form of cir-cularity discussed in the first section, i.e. the idea that the provability conditions are described by an infinite conjunction. We are forced to consider ``na\"ive'' the view accord-ing to which a proof of a universally quantified statement is a family of arbitrary proofs of all of its substitution instances: the \emph{Hauptsatz} forces us to accept that this family of proofs must satisfy some uniformity requirement.

For instance, one might add to the explanation of the second order quantifier provided by the infinitary rule below
$$
\AXC{$\dots$}
\AXC{$\D D_{A}$}
\noLine
\UIC{$A[B/X]$}
\AXC{$\dots$}
\RL{$B$ prop}
\TrinaryInfC{$\forall XA$}
\DP
$$
the demand that the family of derivations $\D D_{B}$ satisfies a uniformity condition expressing the fact that the derivations $\D D_{B}$ are ``the same'', independently of the proposition $B$. This can be expressed by the condition that the equation below
$$
\D D_{X}[B/X] \ \equiv \ \D D_{B}
$$
holds for all proposition $B$, where the symbol denotes syntactic equality between derivations. This uniformity condition states that the family of derivations $(\D D_{B})_{B \ prop}$ is uniform in $B$. The uniformity condition offers then a way-out of the circularity problem: the derivation $\D D_{B}$, for $B$ of arbitrary logical complexity, can be constructed starting from the ``finite'' derivation $\D D_{X}$ , whose conclusion has a logical complexity strictly bounded by that of $\forall XA$, simply by replacing, in the latter, all occurrences of $X$ by $B$.

Hence, even if $\forall XA$ can still be instantiated as a value of its variable $X$ (i.e. even if the $\B{VCP}$ is violated), no vicious circle occurs \emph{at the level of provability conditions}.

This simple syntactical condition should be compared with the mutilation condition and the naturality condition described in appendix \ref{app::A}, which are technical conditions used in the categorial interpretation of polymorphism. All these uniformity conditions share the fact that the potentially circular cases must be computable starting from the non circular (or ``finite'') ones, i.e. those which respect the complexity constraint.

At first glance, it seems quite tempting to say that the uniformity of polymorphism provides a mathematical vindication of Carnap's argument, stating that a proof of a universally quantified proposition $\forall XA$ does not amount to a proof of all of its possible substitution instances, but rather to a proof of $A[X]$ in which the variable $X$ occurs as a parameter. In \cite{Longo1997} it is argued, for instance, that Carnap's argument is vindicated by a particular way of formalizing parametric polymorphism, i.e. by means of the technical notion of \emph{genericity} (\cite{Longo1993}), stating, roughly, that polymorphic programs which are indistinguishable on one of their types, are indistinguishable on all of their types. A reference is also made to a remark by Herbrand in \cite{Goldfarb1987}, stating that proofs of universal statements (in general, not only second order) can be seen as ``prototype proofs'', which can be used schematically to produce proofs of all substitution instances.

\begin{quote}\small
Indeed, we want to argue that these impredicative constructions are \emph{safe}, along the lines of Carnap's argument, by showing that System $\B F$ contains a precise notion of prototype proof, in the sense of Herbrand, and of generic types, with strong ``coherence'' properties. \cite{Longo1997}
\end{quote}

Many different mathematical characterizations of parametric, or uniform, polymorphism can be found in the literature, and actually this constitutes a still lively field of research in theoretical computer science. Reynolds' \emph{relational parametricity} (\cite{Reynolds1983}) has probably been the most investigated theory of polymorphism in the last thirty years (see \cite{Hermida2014} for a reconstruction); other important examples are Bainbridge's \emph{functorial polymorphism} \cite{Bainbridge1990, Girard1992} and Girard's theory of \emph{stable functors} \cite{Girard1986}. It must be admitted that, with very few exceptions (like \cite{Longo1997}), the huge literature on polymorphism of the last forty years has gone practically unnoticed in the philosophical debate on impredicativity. It is out of the scope of this paper to make a complete survey of all of these advances. The reader interested in the mathematical counterpart of the problem of impredicativity will find in appendix \ref{app::A} a brief reconstruction of some key aspects in the theory of stable functors and, more generally, in the functorial approach to polymorphism. These approaches provide a good example of the technical challenges constituted by the mathematical interpretation of impredicative type disciplines and of the non trivial solutions offered.

\section{The type of all types.}\label{sec::3} The problem of circularity \emph{in the elimination} of redundant impredicative concepts in a proof is investigated by comparing two arguments: an argument for the normalization of Martin-L\"of's impredicative type theory (\cite{ML70}) and Girard's normalization proof, already discussed in the previous section. A particular form of circularity, that we call \emph{reflecting circularity}, will be distinguished from the vicious circularity seen in Frege's argument. A conjecture will be made to distinguish between the harmful reflecting circularity in the first argument and the apparently harmless reflecting circularity in the second.

\subsection{Martin-L\"of's proof.} The fundamental idea of Martin-L\"of's original 1971 intuitionistic type theory (\cite{ML70}), $\B{ITT}_V$ for short, arises naturally from the conjunction of the following three demands: first, Russell's principle $\B{T1}$ that the range of application of a function must form a type; second, the possibility of quantifying over propositions, that is, the demand that the system contains System $\B F$ polymorphism; third, a strict interpretation of the propositions-as-types correspondence, that is, the identification of every proposition with a type and, viceversa, of every type with a proposition.

As we already remarked in the first section, the conjunction of Russell's principle $\B{T1}$ with the demand that quantification over propositions be possible implies the existence of a type of all propositions. Now, the complete identification of propositions and types makes the type of all propositions the type of all types: Martin-L\"of's type theory is a dependent type theory with a type $V$ of all types, i.e. endowed with the axiom
$$V\in V$$

This axiom is clearly incompatible with the $\B{VCP}$ and with Russell's ban on self-application:
$€V$ can appear as a value of any variable which has type $V$.

The theory $\B{ITT}_V$ is extremely simple to formulate, as it essentially contains two constructions: a dependent product $(\Pi x\in \sigma)\tau$ and the type $V\in V$.

In order to appreciate how these two simple constructions allow to build an extremely powerful impredicative type theory, much more expressive than System $\B F$, let us first recall the notion of dependent type.

Dependent types were introduced by De Bruijn (\cite{bruijn70}); the basic idea is that, given a type $\sigma$
and a type $ \tau(x)$, depending on a variable $x\in \sigma$, one can construct the dependent product $(\Pi x\in \sigma)\tau(x)$, which is the type of all functions associating, with each $x\in \sigma$, an object of type $\tau(x)$. The terms of type a dependent product are constructed, similarly to the case of implication types, by an abstraction and an application operation: if, on the assumption that $x\in \sigma$, the term $M$ has type $\tau(x)$, then one can form the abstracted term $(\lambda x\in \sigma)M$ of type $(\Pi x\in \sigma)\tau(x)$. Conversely, given terms $M$ and $N$ of type, respectively, $(\Pi x\in \sigma)\tau(x)$ and $\sigma$, one can form the application term $MN$ which has type $\tau(N)$.

All types of System $\B F$ can now be obtained in $\B{ITT}_V$ as special cases of the product type. In particular, the type $\sigma\to \tau$ corresponds to the product $(\Pi x\in \sigma)\tau$ where $\tau$ does not depend on $x$, and the type $\forall X\sigma$ corresponds to the product $(\Pi x\in V)\sigma$ over all types. For instance, the type $\forall X(X \to X)$ of System $\B F$ can be written as the product $(\Pi x\in V)(\Pi y\in x)y$ or, equivalently, as $(\Pi x\in V)(x\to x)$.

In $\B{ITT}_V$ any type $\sigma$ occurring in righthand position in a judgment, say $M\in \sigma$ , can also appears as a term, i.e. in lefthand position, in the judgement $\sigma\in V$. In $\B{ITT}_V$ there is no meta-theoretical distinction between types and terms: the fact that a term $a$ denotes a type is expressed, \emph{in the theory}, by the judgement $a\in V$. Indeed, the system $\B{ITT}_V$ combines and unifies the idea that every proposition must be assigned a type (the type prop of propositions, $V$ in this case) and the idea that every proposition is identified with a type: for instance, the impredicative proposition $\forall X(X \to X)$ is well-typed (of type $V$) and is a type (as any object of type $V$).

A strongly related system is Girard's System $\B U^{-}$ (\cite{Girard72}). This is an extension of System $\B F$ in which the type discipline of System $\B F$ is essentially used to type proposition. In System $\B U^{-}$, types are not identified with propositions; rather propositions constitute a special class of types, those of type $prop$. However, this theory is extremely impredicative, due to the existence of impredicative universes (corresponding to impredicative types) whose elements can be seen as polymorphic proposition constructors. System $U^{-}$, which can be embedded in $\B{ITT}_V$ , was proved inconsistent by Girard's paradox (see below).

The theory $\B{ITT}_V$ is much stronger than System $\B F$: in addition to quantification over types, one can explicitly define functions of type $V \to V$, i.e. functions over types (for instance the power-set function $\B P = (\lambda x\in V)(x\to V)$), functions over functions over types (i.e. of type $(V\to V)\to V$ ), and so on; one can further extend this hierarchy by considering functions of an impredicative type. For instance the type $(\Pi x\in V)(x\to V)$ can be taken as the type of all functions from objects of an arbitrary type to types.

In spite of the overwhelming impredicativity of $\B{ITT}_V$ , Martin-L\"of was able to devise a notion of computability for this system, by generalizing Girard's reducibility candidates technique in a very elegant way. With this notion, in a paper which remained unpublished (\cite{ML70}), he provided a normalization argument for $\B{ITT}_V$ , purported to show the consistency of the theory. However, $\B{ITT}_V$ is not consistent: a not normalizing term of type $\bot= (\Pi x\in V)x$ was found by Girard in 1971, by exploiting a variant of Burali-Forti paradox. Girard's paradox was actually constructed for System $\B U^{-}$. Girard's paradox is much more complex than Russell's paradox (for a reconstruction, see \cite{Coquand1986, Hurkens95}). Indeed,
the circularity at work in $\B{ITT}_V$ , as we will try to show, is much harder to identify than that of Frege's system.

In what follows we sketch a definition of reducibility candidates for $\B{ITT}_V$ , inspired from Martin-L\"of's argument as well as from Girard's work on System $\B U^{-}$. The aim of this presentation is not to reconstruct in full detail the normalization argument, but rather to show that a definition of reducibility candidates for this extremely impredicative theory can be given without falling in the tough circularity of Frege's stipulations.

By comparing this argument with the normalization proof of System $\B F$, we will try to understand what actually goes wrong in the case of $\B{ITT}_V$ . In particular, it will be argued that a precise answer might depend on the solution of a conjecture concerning the expressive power of $\B{ITT}_V$ , which will be discussed in the next subsections.

Let $\Lambda$ be the set of all terms of $\B{ITT}_V$ . When no ambiguity occurs, we will indicate by small letters $a,b,c,\dots$ the terms of occurring at the lefthand side of a judgement (i.e. in term position) and by capital letters $A,B,C,\dots$ the terms of occurring at the righthand side (i.e. in type position).

A reducibility candidate will be a pair $(s,C)$ where $s$ is a set and $C\in \wp(\Lambda \times s)$ is a relation over terms and elements of $s$. Intuitively, one can think of this relation as a ``candidate'' for a relation of computability expressing the property ``$\xi$ is a computation of $a$''. Indeed, we'll demand that, from the fact that $(a,\xi)$ is in $C$, it follows that $a$ is normalizable.

An interpretation of the terms of	 $\Lambda$ obtained as follows:

\begin{itemize}

\item to every type $A(x_{1},\dots, x_{n})$ (i.e. every term of type $V$), depending on variables $x_{1}\in A_{1}, x_{2}\in A_{2}(x_{1}), \dots, x_{n}\in A_{n}(x_{1},\dots, x_{n-1})$, a pair $\alpha_{A}(\OV \xi)=\big ( s_{A}(\OV \xi), Comp_{A}(\OV \xi)\big)$, depending on an assignment $\OV \xi= \xi_{1}\in s_{A_{1}}, \xi_{2}\in s_{A_{2}}(\xi_{1}), \dots, \xi_{n}\in s_{A_{n}}(\xi_{1},\dots, \xi_{n-1})$, is associated, where $s_{A}$ is a class and $Comp_{A}(\OV \xi)(a,\eta)$ is the \emph{predicate} of computability of type $A$, which is a predicate over terms of type $A$ and objects of $s_{A}$ expressing the property ``$\eta$ is a computation of $a$, under the assignment $\OV \xi$'';

\item to every term $a(x_{1},\dots, x_{n})$ of type $A(x_{1},\dots, x_{n})$ depending on variables $x_{1}\in A_{1},x_{2}\in A_{2}(x_{1}),\dots, x_{n}\in A_{n}(x_{1},\dots, x_{n-1})$, an object $\alpha_{a}(\OV \xi)$ belonging to the class $s_{A}(\OV \xi)$, depending as above on an assignment $\OV \xi$, is associated ($\alpha_{a}(\OV \xi)$ represents, intuitively, the computation of $a$).

\end{itemize}

In particular, if $A$ is the type $V$ , then $s_{A}$ is the class of all reducibility candidates and $Comp_{V}(A,(s,C))$ expresses the property ``$A$ is normal and $(s,C)$ is a pair such that, for all term $a\in A$ and all $\xi \in s$, $C(a,\xi)$ implies that $a$ is normalizable''. The impredicativity of $V$ as the type of all types is thus reflected by its associated computability predicate, which expresses the property of being a reducibility candidate.

As the product type $A = (\Pi x\in B)C(x)$ generalizes and unifies implication types (when
$C$ does not depend on $x$) and polymorphic types (in the case in which $A = V $), the definition of $\alpha_{A}(\OV \xi)$ in this case generalizes and unifies Tait's and Girard's definitions of computability for implication and polymorphic types: $\alpha_{A}(\OV \xi)$ is the pair $\big (s_{A}(\OV \xi), Comp_{A}(\OV \xi)\big )$ where
\begin{itemize}
\item $s_{A}(\OV \xi)$ is the type of all functions $\eta$ associating, with any $\xi'\in s_{B}(\OV \xi)$, an element $\eta(\xi')\in s_{C}(\OV\xi\xi')$;

\item $Comp_{A}(\OV \xi)(a,\xi')$ holds when, for every term $b$ and every $\eta\in s_{B}$, if $Comp_{B}(\OV \xi)(b,\eta)$ holds, then $Comp_{C}(\OV \xi \xi')(ab, \xi'(\eta))$ holds.
\end{itemize}

Let us see in what sense the definition generalizes Tait's and Girard's ones:
\begin{itemize}

\item[1.] 
if the type $C(x)$ does not depend on $x$, i.e. if $A = B \to C$, then $\xi'$ is a computation of $a$ of type $B\to C$ (under an assignment $\OV \xi$) when for every term $b$ and $\eta\in s_{B}(\OV \xi)$ such that $\eta$ is a computation of $b$ (under $\OV \xi$) of type $B$, then the application $\xi'(\eta)$ of $\xi'$ to $\eta$ is a computation of $ab$ of type $C$ (under $\OV \xi$);

\item[2.] if $B$ is the type $V$, i.e.  if we can write $A$ as $\forall XC[X]$, then $\xi'$ is a computation
of $a$ of type $\forall XC[X]$ (under an assignment $\OV \xi$) when, for every type $B$ and pair $(s, C)$, if $(s,C)$ is a reducibility candidate for the type $B$ (i.e. if $Comp_{V}(B, (s,C))$ holds), then $\xi'(s,C)$ is a computation of $aB$ (under $\OV \xi(s, C)$). In other words, $\xi'$ is a computation of $a$ (under an assignment $\OV \xi$) if, for every type $B$ and reducibility candidate $(s, C)$ for that type, the computation $\xi'(s, C)$ is a computation for $aB$ (under the assignment $\OV \xi(s, C)$ in which the variable $X$ is associated to $(s, C)$).
\end{itemize}

In order to show that every term is computable, one can now argue that, for every term $a(x_{1},\dots, x_{n})$ of type $A(x_{1},\dots, x_{n})$ and assignment $\OV \xi$, $\alpha_{a}(\OV \xi)$ is a computation of $a$
of type $A$ (under $\OV \xi$), i.e. $Comp_{A}(\OV \xi)\big (a, \alpha_{a}(\OV \xi)\big )$ holds. In particular, since $A$ has type $V$,
this implies that $Comp_{V}\big (A,(s_{A},Comp_{A})\big )$ holds, which means that, if $Comp_{A}(\OV \xi)(a,\xi')$ holds for some $\xi'$, then $a$ is normalizable.
From $Comp_{A}(\OV \xi)(a, \alpha_{a}(\OV \xi))$ it follows then that $a$ 
is normalizable. The argument can be given by induction
on the
derivation of a judgement $a(x_{1},\dots, x_{n})\in A(x_{1},\dots, x_{n})$ under the hypotheses $x_{1}\in A_{1},\dots, x_{n}\in A_{n}(x_{1},\dots, x_{n-1})$. Some steps of this argument will be made explicit in the next subsections.

The problem in the argument just sketched cannot reside, like for Frege, in the induction hypotheses in the definition of computability: the predicates of computability are defined by a well-founded induction over types. In the case of $V$ the predicate $Comp_{V}(a,\xi)$, expressing the property of being a reducibility candidate, is defined without reference to other computability predicates. In the case of an impredicative type $\forall XB =  (\Pi x\in V)B$, the computability predicate $Comp_{\forall XB}$ is defined in terms of a (countably) infinite class of computability predicates $Comp_{B}$ of types of strictly smaller complexity. The predicate $Comp_{\forall XB}$ can be expressed as an infinite conjunction by the
rule below:

$$
\AXC{$\dots$}
\AXC{$Comp_{B}(\OV \xi(s,C))(cA, \eta((s,C)))$}
\AXC{$\dots$}
\RL{for any $A\in V$, $(s,C)\in s_{V}$ s.t. $Comp_{V}(A,(s,C))$}
\TrinaryInfC{$Comp_{(\Pi x\in V)B}(\OV \xi)(c,\eta)$}
\DP
$$

which is well-founded, because in all premisses there appear computability predicates on strictly simpler types and, as we have just seen, the verifications of computability for V is always well-founded.

Nevertheless, by inspecting the stipulations defining the notion of computability, a set-theory closely imitating the type structure of $\B{ITT}_V$ is presupposed: on the one had, the predicate $Comp_{V}$ is defined over the class of all reducibility candidates; on the other hand, the axiom $V\in V$ requires that $Comp_{V}\big (V, (CR, Comp_{V})\big )$ holds, where $CR$ denotes the class of all reducibility candidates: hence $(CR, Comp_{V})$ is a reducibility candidate, which implies that the class $CR$ must be a \emph{set}, a set containing, in particular, the pair
$(CR, Comp_{V})$. Clearly, neither $\B{ZF}$ nor its most well-known consistent extensions contain such a set.

Hence, on the one hand the stipulations provided for computability do not involve vicious circles: while Frege's notion of having a denotation would not be well-defined in any universe, computability for the type of all types would be well-defined in a universe containing a set of all sets. On the other hand, reasonably, if a set-theory can be devised in which computability can be formalized, then in this theory a false proposition express-ing the fact that $\B{ITT}_V$ is consistent should be derivable. Since such a theory should plausibly be able to derive the true proposition that $\B{ITT}_V$ is not consistent (through a formalization of Girard's paradox), the set-theory in question should be inconsistent.

\subsection{Reflecting circularity in the normalization of System $\B F$.} In the last sub-section we sketched an argument for the normalization of $\B{ITT}_V$ and we argued that the circularity in it is not the vicious circularity consisting in the definition of a notion which presupposes that very notion to be already defined, but rather in the circularity arising from the fact that the whole argument must be formalized in a (meta-)theory which reflects the (impredicative) structure of the theory.

This form of circularity, that we will call \emph{reflecting circularity}, can be described in more clear terms if we go back to Girard's normalization proof and to its most surprising part, the substitution lemma. This lemma says that the computability of type $\sigma$ relative to an assignment $s_{X\mapsto Comp_{\tau}[s]}$ which assigns to $X$ the computability of type $\tau$ is the same as the computability of type $\sigma[\tau/X]$ relative to $s$. This is the part of the proof which essentially validates the extraction rule, i.e. the elimination rule of the universal quantifier, as it implies that a term computable of type $\forall X\sigma$ relative to $s$, hence computable of type $\sigma$ relative to all assignments  $s_{X\mapsto \C E}$, is in particular computable of type $\sigma[\tau/X]$ relative to $s$.

The proof of the lemma can be decomposed in two parts: first one has to show that the predicate of computability of type $\tau$ (relative to $s$) defines a reducibility candidate (hence, that it defines a set); then one proves that the computability of type relative to $s_{X\mapsto Comp_{\tau}[s]}$ is the same that the computability of type $\sigma[\tau/X]$ relative to $s$.

While the second part can be carried over by induction over types in a rather straight-forward way, the first part cannot be proved by induction, as it rests on the use of an instance of the comprehension schema. Indeed, the fact that computability predicates define sets cannot be proved by induction, precisely because of the definition of computability for universally quantified types: stating that a term $M$ is computable of type $\forall X\sigma$ relative to $s$ when, for every type $\tau $ and candidate $\C E$ for $\tau$, $M$ is computable at type $\sigma$
relative to $s_{X\mapsto \C E}$, corresponds to defining a predicate by an intersection over a family of computability predicates indexed by all candidates for $\tau$ (where $\tau$ can be of arbitrary logical complexity). Hence one cannot prove that computability at type $\forall X\sigma$ is a candidate (hence a set) starting from the ``induction hypothesis'' that all computability predicates
of type $\sigma$ relative to $s_{X\mapsto \C E}$ are candidates: by taking $\tau$ to be $\forall X\sigma$ the candidate we are defining appears among the candidates of the induction hypothesis! In a word, we would get back to the regress of Frege's argument.

However, that the computability of type $\forall X\sigma$ is a candidate (a candidate depending on \emph{all} candidates, including himself), can be stated directly by an instance of the com-prehension schema of set theory\footnote{More precisely, by an instance of the restricted comprehension schema, as $Comp_{\forall X\sigma}$ is a subset of the set of all terms of type $\forall X \sigma$.
}: the $\B{VCP}$ is clearly violated, but infinite intersections of candidates indexed by all candidates can be shown to exist in $\B{ZF}$, as candidates are subsets of the power-set of the set of all $\lambda$-terms (see \cite{Gallier}).

The passage which justifies the impredicative elimination rule of System $\B F$ relies thus on an impredicative set-theoretical construction at the level of reducibility candidates. In some sense, the set-theory in which reducibility candidates are formalized closely imitates the structure of System $\B F$ types. Hence, also the normalization argument of System $\B F$ manifests some form of reflecting circularity.

We can make this remark more precise by exploiting the fact that the theory of reducibility candidates can be expressed directly in the language of second order Peano Arithmetics $\B{PA}^2$, without relying on set-theory, as the following hold:

\begin{enumerate}
	\item the definitions of reducibility candidates and computability predicates can be expressed in the language $\C L^{2}$of $\B{PA}^2$. More precisely
	\begin{itemize}
\item there exists a predicate $\TT{CR}[X]$ of $\C L^{2}$, depending on a predicate variable $X$, expressing the property ``$X$ is a reducibility candidate'';

\item for every type $\sigma$ there exists a predicate $\TT{Comp}_{\sigma}[X_{1},\dots, x_{n}](x)$ of $\C L^{2}$ (where contains exactly $n$ free type variables) expressing the property ``$x$ is the code of a term computable of type $\sigma$ relative to the assignment $X_{1},\dots, X_{n}$''. For instance, for $\sigma= \forall X\rho$ , the predicate $\TT{Comp}_{\sigma}$ can be defined inductively as follows

 $$
 \TT{Comp}_{\forall X\rho}[X_{1},\dots, X_{n}](x) := \forall Z\big ( \TT{CR}[Z] \To \TT{Comp}_{\rho}[X_{1},\dots, X_{n},Z](x) \big )
 $$
\end{itemize}

\item since the consistency of System $\B F$ implies that of $\B{PA}^2$, the normalization proof can-not be entirely formalized in $\B{PA}^2$, because of G\"odel's second theorem (the argument is developed in detail in \cite{ptlc1});

\item however, locally, i.e. for every term of System $\B F$, the argument showing that a term is computable of its type and, hence, normalizable, can be formalized in $\B{PA}^2$. In particular, for every term $M$ of type $\sigma$, it can be proved in $\B{PA}^2$that $\TT{CR}[X_{1}] \land \dots \land \TT{CR}[X_{n}] \To \TT{Comp}_{\sigma}[X_{1},\dots, X_{n}](\sharp M)$ (where $\sharp M$ indicates the code of $M$).
\end{enumerate}

The circularity occurring in the first part of the proof of the substitution lemma can be now exposed by formalizing the argument for the computability of a term within $\B{PA}^2$: the justification of the extraction/$\forall$E rule
$$
\AXC{$M\in \forall X\sigma$}
\RL{$Ext$}
\UIC{$M\{\tau\}\in \sigma[\tau/X]$}
\DP
$$
is formalized by an application of $\forall$E in $\B{PA}^2$, i.e. (essentially) \emph{the same rule}, applied to
computability predicates (expressing the fact that $\TT{Comp}_{\tau}$
is a reducibility predicate):
$$
\AXC{$\TT{Comp}_{\forall X\sigma}[\vec X](n)$}
\RL{$\forall$E}
\UIC{$\TT{CR}[\TT{Comp}_{\tau}]\To \TT{Comp}_{\sigma}[\vec X, \TT{Comp_{\tau}}](n)$}
\DP
$$

In a word, comprehension applied to the type $\tau$ is justified by comprehension applied to the predicate $\TT{Comp}_{\tau}$. Here's how Girard comments on this aspect:

\begin{quote}\small
Speaking of circularity, take for instance comprehension: this schema is represented by extraction, but the reducibility of extraction requires comprehension, roughly the one under study. [...] If one carefully looks at the proof of reducibility for System $\B F$, one discovers that the reducibility of type $A$ closely imitates the formula $A$. Which makes that the extraction on $B$ - the only delicate point - is justified by a comprehension on something which is roughly $B$. \cite{Girard1989}
\end{quote}

\subsection{Reflecting circularity in the normalization of Martin-L\"of's system.}\label{subs::3.3} The normalization argument for System $\B F$ with the reducibility candidates technique exploits a set-theoretical translation of the rule of extraction (i.e. of the $\forall$E logical rule). The normalization argument sketched for $\B{ITT}_V$ generalizes the reducibility candidates tech-nique and exploits a set-theoretical universe closely imitating the type structure of the theory. Both arguments manifest a form of what we called reflecting circularity. In order to compare the forms of reflecting circularity at work in them, we reconstruct the part of the argument for $\B{ITT}_V$ which plays a role analogous to the substitution lemma in Girard's proof.

Our presentation of the argument highlights the distinction between computability predicates of the form $Comp_{A}(\OV \xi)$ and reducibility candidates $(s,C)$, in which the relation $C$, a set, is a ``candidate'' for being a computability predicate. Due to the existence of a type of all types, the role of Girard's substitution lemma (whose relevant impredicative step is the one in which it is affirmed that the stipulations defining a computability predicate actually define a set, hence a reducibility candidate) is internalized by the computability predicate $Comp_{V}$: in order to justify judgments of the form $A(\vec X)\in V$ one
has to verify that $Comp_{V}\big (A, (s_{A}(\OV \xi), Comp_{A}(\OV \xi))\big )$ holds, i.e. that the pair $(s_{A}(\OV \xi), Comp_{A}(\OV \xi))$ is a reducibility candidate. In a word, $Comp_{A}(\OV \xi)$, a predicate, becomes a set.

Computations of the form $Comp_{V}\big (A,(s_{A}(\OV \xi), Comp_{A}(\OV \xi))\big )$ occur in two cases in the proof: in the justification of the axiom $V\in V$ and in the justification of the elimination rule for the product type, which corresponds to the application operation, in the case of a product over $V$:

$$
\AXC{$a\in (\Pi x\in V)B(x)$}
\AXC{$C\in V$}
\RL{$\Pi$E}
\BIC{$aC\in B(C)$}
\DP
$$

Similarly to what has been shown for System $\B F$, the definitions of computability predicates and reducibility candidates presented for $\B{ITT}_V$ can be expressed in the language of $\B{ITT}_V$ and the normalization argument, for a given term $a\in A$, can be simulated in the theory. More precisely,
\begin{itemize}

\item there exists a predicate $\TT{CR}(s,C)$ of type $V\to V\to V$ expressing the property ``$(s, C)$ is a reducibility candidate'';

\item for every type $A(\vec x)$ there exists a predicate $\lambda u.\lambda v. \TT{Comp}_{A}(\vec z)(u,v)$ of type $\B{N}\to \tau_{A}(\vec z)\to V$ (depending on variables $\vec z$), where $\B N$ is the type of integers $(\Pi x\in V)((x\to x)\to (x\to x))$, $\tau_{A}(\vec z)$ is a type, depending on $\vec z$, representing the set $s_{A}(\OV \xi)$, expressing the property ``$v$ is a computation of the term coded by $u$ under the assignment $\vec z$'';

\item for every term $a(\vec x)\in A$, there exists a term $\boldsymbol \alpha_{a}(\vec z)\in \tau_{A}(\vec z)$ as well as a term of type $\TT{Comp}_{A}(\vec z)(\sharp a, \boldsymbol \alpha_{a}(\vec z))$.

\end{itemize}

We can reconstruct the form of the predicate $\TT{Comp}_{A}(\vec z)$ in the two most significative cases:
\begin{description}
\item[$i.$:] $\TT{Comp}_{V}(u,v)$ is the conjunction of the predicates expressing the properties ``$u$ is the
code of a type in normal form'' and $\TT{CR}(v)$, respectively;

\item[$ii.$:] $\TT{Comp}_{(\Pi x\in A)B(x)}(\vec z)(u,v) := 
(\Pi m\in \B N)(\Pi w\in \tau_{A})(\TT{Comp}_{A}(\vec z)(m,w) \To 
\TT{Comp}_{B}(\vec zw)(\TT{App}(n,m), vw))$, where $\TT{App}(x,y)$ is a function coding application, i.e. such that $\TT{App}(\sharp a, \sharp b)=\sharp ab$.
\end{description}

The two steps in the normalization argument discussed above can then be simulated inside $\B{ITT}_V$ . First, in the case of the axiom $V\in V$, the justification corresponds to constructing a term of type $Comp_{V}(\sharp V, (\boldsymbol \alpha_{V}, \TT{Comp}_{V}))$, that is, proving that $\sharp V$ is the code of a normal term and that $\TT{CR}((\boldsymbol \alpha_{V}, \TT{Comp}_{V}))$ holds. However, the delicate part comes before the proof of these facts: in order to type the expression $\TT{Comp}_{V}(\sharp V,(\boldsymbol \alpha_{V}, \TT{Comp}_{V}))$, we must show that $\TT{Comp}_{V}$ has type $\B N\to V \to V$: for this the axiom $V\in V$ must then be invoked several times.

In the case of the elimination rule $\Pi$E, we must exploit two occurrences of the same rule, along with the usual ``comprehension axiom'', which is just the fact that $\TT{Comp}_{C}\in V$ (the derivation is displayed in figure \ref{fig::1}).

Similarly to the normalization argument for System $\B F$, the steps of the computability argument for a term can be simulated in the theory itself, and the justification of an impredicative rule is represented by one or more occurrences of the same rule.

As we already mentioned, in accordance with G\"odel's second theorem, the normal-ization argument for System $\B F$ cannot be entirely formalized in $\B{PA}^2$; in particular, the notion of computability is not definable in the language $\C L^{2}$: there exists no second order formula in the language of arithmetics $\TT{Comp}(\sharp A,y)$ such that, for all $A$, $\TT{Comp}(\sharp A,y)$ is equivalent to $\TT{Comp}_{A}(y)$ (which is definable).

The notion of computability behaves thus in a similar way to the model-theoretic notion of validity: by Tarski's theorem we know that the set of (codes of) valid second order formulas is not definable by a second order formula, though the set of codes of valid second order formula of fixed complexity is definable.

However, as $\B{ITT}_V$ is inconsistent\footnote{A note for the reader who went through appendix \ref{app::A}. It must be observed that, though the theory $\B{ITT}_V$ is inconsistent, there exists an interesting literature on its denotational models (see \cite{Coquand1989} for a reconstruction). This should not surprise the reader: in denotational semantics one is interested in the interpretation of the computation content of programs, a question which is, by itself, independent from the fact that a program represents a valid proof. A logically inconsistent programming languages can be nevertheless sound from a computational perspective. Hence, for instance, interesting denotational models of inconsistent typed $\lambda$-calculi can be constructed, by relying on the existence of domains satisfying equations of the form $X=X\to X$ and, in general, fixed point equations of the form $X=\Phi (X)$, where $\Phi$ is a functor over the category $\B{Dom}$ of domains (see for instance \cite{Coquand1989, Longo91}).
} , G\"odel's second theorem cannot be invoked in this case. This is where a difference might well exist between the two arguments: there is no apparent obstacle to express, in the language of $\B{ITT}_V$ :

	\begin{enumerate}
	\item 	a function  $\tau(x)\in \B N \to V$ such that, for all type $A$, $\tau(\sharp A)=\tau_{A}$;

	\item 	a predicate $(\lambda x\in \B N)(\lambda y\in \B N)(\lambda z\in \tau(x))\TT{Comp}(x,y,z)\in (\Pi x\in \B N)(\B N\to \tau(x)\to V)$ such that, for all type $A$, $\TT{Comp}(\sharp A, y,z)$ is equivalent to the predicate $\TT{Comp}_{A}(y,z)$ (which has type $\B N\to \tau_{A}\to V$).
\end{enumerate}

Technically speaking, this possibility arises from the fact that in $\B{ITT}_V$ , contrarily to $\B{PA}^2$, predicates can be defined by induction over codes of formulas (this fact is shown in appendix \ref{app::B}). Hence, given a suitable coding, the clauses defining computability
(which are given by induction over formulas) can be translated into clauses defining the function $\tau$ and the predicate $\TT{Comp}$ by induction over the codes of formulas.

Given these remarks, it does not seem unreasonable to conjecture that a normalization argument for $\B{ITT}_V$ might well be formalizable in the theory itself\footnote{This means constructing a program of type $(\Pi x\in \B N)(\TT{Typable}(x) \to \TT{Norm}(x))$, where the (recursive) predicate $\TT{Typable}(x)$ expresses the fact that $x$ is the code of a well-typed term and the (recursive) predicate $\TT{Norm}(x)$ expresses the fact that $x$ is normalizable.
}. In a word, $\B{ITT}_V$ might well prove its own consistency in a non trivial way (i.e., without deriving it from the proof of a contradiction, as, by Girard's paradox, contradictions are derivable in $\B{ITT}_V$ ).

\section{Harmless and harmful circularities.} In this paper thee consistency arguments for three impredicative logical systems were discussed. Only one of these argument can be considered as a correct proof. Apparently, the second and the third argument exploit a form of circularity, that we styled reflecting circularity, which is not the vicious circularity diagnosed in the first argument. Rather, it is the circularity of justifying the rules of the system by exploiting the same rules in the argument. However, while the second system, System $\B F$, is consistent, the third system, Martin-L\"of's system $\B{ITT}_V$ , just like Frege's system, is not consistent.

Hence, three different forms of circularity should be distinguished: a vicious one, a not vicious, reflecting and apparently harmless one, and a not vicious, reflecting, though harmful, one. The most interesting challenge is then to provide a clear distinction between the last two.

Obviously, a difference exists between the normalization arguments for System $\B F$ and the one for $\B{ITT}_V$ : we trust the normalization argument for System $\B F$ because it can be formalized in standard set-theory $\B{ZF}$, a theory which we believe to be consistent. On the contrary, the normalization argument for System $\B{ITT}_V$ cannot be formalized in set-theory (and, as we argued, plausibly, cannot be formalized in any consistent set-theory). In other words, while we surely have no reason to trust the second argument, as we are aware of a contradiction derivable in $\B{ITT}_V$ , we have external reasons to trust the first argument, as we are not aware of any contradiction derivable in $\B{ZF}$ and we are confident that no one will be found.

One might object that the alleged consistency of $\B{ZF}$ is only an external reason for our trust in the apparently ``good'' impredicativity of System $\B F$, with respect to the ``bad'' impredicativity of $\B{ITT}_V$ , and sheds no light on the different forms of circularities involved.

The distinction between vicious and reflecting circularity reminds a notorious distinc-tion made by Dummett between two different ways in which an argument for the justi-fication of a logical law can be blamed of circularity: on the one side he considers

\begin{quote}\small
[...] the ordinary gross circularity that consists of including the conclusion to be reached among the initial premisses of the argument. \cite{Dummett1991}
\end{quote}

On the other side, he considers arguments that purport

\begin{quote}\small
to arrive at the conclusion that such-and-such a logical law is valid; and the charge is not that this argument must include among its premisses the statement that the logical law is valid, but only that at least one of the inferential steps in the argument must be taken in accordance with that law. We may call this a ``pragmatic'' circularity. \cite{Dummett1991}
\end{quote}

A pragmatically circular argument is one which employs the rule it is up to justify. For instance, an argument for the justification of the rule of modus ponens (i.e. ($\To$E)) is ``pragmatically circular'' if it employs somewhere an instance of the rule of modus ponens.

While everybody would agree that a viciously circular proof bears no epistemic value, according to Dummett, a ``pragmatically'' circular proof might bear some epistemic value:

\begin{quote}\small
[...] if the justification is addressed to someone who genuinely doubts whether the law is valid, and is intended to persuade him that it is, it will fail of its purpose, since he will not accept the argument. If, on the other hand, it is intended to satisfy the philosopher's perplexity about our entitlement to reason in accordance with such a law, it may well do so. \cite{Dummett1991}
\end{quote}

Coherently with his views on the $\B{VCP}$ and on the meaning of universal quantification, Dummett claims that the justification of second order quantification is a viciously circular one (see \cite{Dummett1991b}, \cite{Dummett2006}). Rather, he consider pragmatic the circularity occurring in the case of the justification of first-order logical rules by means of a normalization argument.

However, our analysis of the normalization arguments highlighted how the circularity involved in the proof-theoretic justification of second order logic is not vicious but re-flecting: the extraction on a polymorphic term, corresponding to the rule $\forall$E of second order intuitionistic logic, is justified by an argument whose only logically complex part (the only impredicative step) can be formalized in $\B{PA}^2$by an occurrence of the same rule $\forall$E.

Hence, on the one hand, reflecting circularity would make a normalization argument powerless in a debate over the legitimacy of second order quantification: a Russellian logician should accept that $\B{VCP}$ be violated before assessing the argument justifying the violation of the $\B{VCP}$. On the other hand, a reflectingly circular argument for the justification of second order logic might satisfy the logician who is already convinced of the legitimacy of the laws of this system.

Moreover, also the circularity involved in the proof-theoretic justification of system $\B{ITT}_V$ is reflecting, rather than vicious: the existence of a type of all types (axiom $V\in V$ ), as well as the extraction on a polymorphic term (rule E) are justified by arguments whose only logically questionable part (the only impredicative step) can be formalized in $\B{ITT}_V$ by the axiom $V\in V$ and the rule $\Pi$E respectively.

The fact that an inconsistent system can be justified on reflectingly circular grounds contradicts the impression of the substantial harmlessness of this form of circularity, that one might get from Dummett's discussion on ``pragmatic'' circularity. Moreover, the distinction between ``pragmatic'' and vicious circularity seems to fail to capture the difference between the apparently harmless circularity of System $\B F$ and the harmful circularity of $\B{ITT}_V$ .

A positive answer to the conjecture proposed in the last subsection might shed some light: if the conjecture were true, that is, if a normalization argument for $\B{ITT}_V$ could be entirely formalized inside $\B{ITT}_V$ , then we would have some definite reason to say that the harmful circularity arising in the normalization argument for $\B{ITT}_V$ is not of the same kind as the circularity in the normalization argument for System $\B F$. Two distinct forms of reflecting circularity could be distinguished:
\begin{itemize}
\item a strong form, corresponding to the fact that the entire normalization argument for a system can be formalized within the system. This circularity, by G\"odel's second theorem, is necessarily harmful;
\item a weak form, corresponding to the fact that the justification argument can be locally (i.e. for every single term or proof) formalized in the system. This circularity is present in the normalization argument for System $\B F$ and, being compatible with G\"odel's second theorem, is not necessarily harmful.
\end{itemize}

In any case, the thesis (advocated by Dummett as well as by others - see section 1) that any justification of second order logic must be viciously circular seems contradicted by a closer look at how the \emph{Hauptsatz} for second order logic is actually demonstrated. The circularity at work in this argument is that of justifying a disputed logical rule by an argument whose only disputable part is represented by the application of instances of that very rule.

Rather, our analysis indicates that a justification of a second order impredicative logical system can be provided which is not circular in the sense in which Frege's justification of the \emph{Grundgesetze} was circular (i.e. viciously circular), but circular in the apparently less harmful sense of justifying impredicative quantification in a way which should be judged logically correct from an impredicative standpoint.

\bibliographystyle{alpha}
\bibliography{Victims}

\appendix

\section{Denotational semantics and the functorial interpretation of polymorphism.}\label{app::A} The brief journey in domain theory and category theory that follows is thought to give the reader a first flavor of the interesting mathematical problems which constitute the technical counterpart of the proof-theoretic problem of impredicativity. It will be suggested that a coherent understanding of the provability conditions for second order logic arises from the functorial interpretations of category theory, where the uni-formity of polymorphism finds a simple and elegant explanation through the notion of \emph{natural transformation}.

\subsection{Continuity and the mutilation condition.} Denotational semantics is the branch of theoretical computer science dedicated to the representation of the mathematical invari-ants of the execution of programs. As its name suggests, at the basis of this discipline stands the ``Fregean'' intuition that the denotation of a program should be independent of the transformations the latter runs through during its execution: for instance, the denotation of the trivial program $\TT{Add}(\TT 5, \TT 7)$ which sums $\TT 5$ and $\TT 7$ should be the same of the result of its execution, i.e. of the output $\TT{12}$.

From a logical point of view, a major difference exists between denotational semantics and model-theoretic semantics: while the latter characterizes provability only, the former provides a semantics of proofs, that is, an interpretation in which proofs are assigned ``denotations'' which are invariant under normalization, i.e. under the transformations which eliminate detours.

The denotational semantics of typed $\lambda$-calculi are usually given by interpreting types as particular mathematical structures, called \emph{domains}, and by interpreting terms as certain subsets of these domains. Intuitively, a domain can be thought as a set of information states, endowed with an order relation, corresponding to the fact of a state being ``more informative'' than another: a state $a'$ is more informative than a state $a$ if the information provided by $a'$ is coherent with that in $a$, and the former has more information than $a$.

A different way to present a domain is as a set of informations plus a relation of coherence between informations. An information state is then defined as a set of coherent informations and the order relation between information states is given by inclusion.

Given a domain $A$, let us call then $coh(A)$ the set of information states of $A$ (i.e. the set of coherent subsets of $A$). Given two domains $A,B$, one can define the domain $A\to B$ whose elements are pairs $(a, z)$ made of a finite information state $\in coh(A)$\footnote{We refer here to a special class of domains, called \emph{coherence spaces}, see \cite{Girard1989}.
} and of a state in $B$. A function $f$ from $coh(A)$ to $coh(B)$ can be thought as the infinite information state of $A \to B$ made of all the pairs $(a, z)$ such that $z \in f(a)$. These pairs are then single pieces of information about $f$. The coherence relation for $A\to B$ is the following: two pairs $(a, z), (a', z')$ are coherent when, anytime the informations in $a$ and $a'$ are (distinct and) pairwise coherent for $A$, then $z$ and $z'$ are (distinct and) coherent for $B$. In particular, if $a'$ is more informative than $a$, then $z'$ cannot be incoherent with $z$. Hence, the information states of $A\to B$ correspond to partial functions from $coh(A)$ to $coh(B)$.

From the fact that, in a pair $(a, z)$, $a$ is a finite information state, it follows that the functions $f$ from $coh(A)$ to $coh(B)$ corresponding to information states in $A \to B$ satisfy a \emph{continuity condition}: the value of $f$ for an arbitrary (possibly infinite) $\in coh(A)$ can be approximated from the values of $f$ on the finite information states contained in
$a$. Indeed, if $z\in f(a)$, for some $a \in coh(A)$, then there exists $a'\subset_{fin}a$ such that $z\in f(a')$\footnote{This property can be stated as the fact that $f(a)$ is the direct union of all $f(a')$, where $a'$ varies among all finite subsets of $a$.
}.

The continuity condition has a natural interpretation in terms of computability: in order to compute the result of applying a program to an input containing a possibly infinite amount of information (think of a real number, i.e. to an infinite list of digits), one cannot exploit more than a finite amount of information about the input. For instance, a function from real numbers to integers giving value $1$ only if all digits of its input are, say, minor than $3$, would not be continuous (nor computable)\footnote{	From a mathematical point of view, the situation is analogous to that of continuous functions in real analysis: since every real number is the limit of a converging series of rationals, the value of a continuous function $f$ over a real number $r$ can be approximated from its ``finite'' approximations, i.e. the values of $f$ over the rationals converging to $r$. Moreover, the fact that continuous functions over the reals are determined by their rational values implies that the cardinality of the set $\C C(\mathbb R)$ of continuous functions over the reals is exactly that of the reals ($|\C C(\mathbb R)| = |\mathbb R|$). Similarly, one can show that, if $A$ and $B$ are countable domains, the set of continuous functions from $A$ to $B$ is countable (contrarily to the set of all functions from $A$ to $B$). This fact is at the heart of the great flexibility of domain theory, which allows to define domains $A$ which are isomorphic to the domain of continuous functions $A \to A$.
}. In a word, programs must correspond to continuous functions over information states.

The interpretation of System $\B F$ in domains requires a delicate abstraction: one has to interpret polymorphic programs, i.e. special functions whose domain is the class of all types. These functions should then correspond to information states over a domain of the form $A \to B$, where $A$ is a domain whose information states are all domains! From the story of $\B{ITT}_V$ we know that a domain of all domains is probably not the right direction to follow. In particular, since domains form a proper class, there can be nothing like a domain (i.e. a set) of all domains. Nevertheless, a way out of this circularity is
provided by category theory: one can define a category of domains and use it as if it were a domain!

Actually, several distinct categories of domains can be defined, depending on the choice of the class of morphisms between domains. For reasons connected with the interpretation of implication, it is convenient to restrict the class of morphisms not to contain all continuous functions, but just injective continuous functions\footnote{More precisely, the standard choice are retraction pairs (see \cite{Girard1986, Coquand1989, Longo91}). The important property is the fact that a morphism $f$ always has an inverse $f^{-1}$. Technically, this restriction makes the interpretation of implication a covariant functor.
}.

The category $\B{Dom}$ of domains can be seen as a special domain itself: finite domains play the role of informations, two finite domains $A,B$ being coherent when there exists a domain $C$ and morphisms $f:A\to C, g:B\to C$. An arbitrary domain can then be seen as the information state made of all its finite subdomains. The order relation between information states is simply given by the morphisms, i.e. injective continuous functions, between domains.

A polymorphic program, say a program of type $\forall X(X \to X)$, should then correspond to a continuous function $f$ associating, with any domain $A$, a morphism from $A$ to $A$; as usual, a problem of circularity arises here, as we are trying to define a new domain (the interpretation of $\forall X(X \to X)$) whose information states correspond to functions defined over all domains, including the one we are defining.

However, a way out of the circularity is provided by the continuity condition: as soon as $f$ is continuous, its values over an arbitrary domain are approximated (and hence univocally determined) by its values over finite domains. Hence the ``circular'' value of $f$
(i.e. the value of $f$ when the argument is the domain of $f$) can be computed starting from the values of $f$ on finite domains, making circularity disappear. We do not enter into the technical details of this delicate construction (the reader can look at \cite{Girard1986, Coquand1989, Longo91}). The important aspect is that the interpretation of $\forall X(X \to X)$ is an infinite domain entirely described in terms of finite domains.

The continuity condition can be considered as a uniformity condition: a term of type
$\forall X(X \to X)$ becomes a family $(t_{A})_{A\in \B{Dom}}$ of morphisms, where for any domain $A$, $t_{A}$ is a morphism from $A$ to $A$; now, continuity implies that the morphism $t_{A}$ is approximated by the morphisms $t_{A'}$, where $A'$ is a finite subdomain of $A$. This means that $t_{A}$ must be the union of all $t_{A'}$, for $A'$ a finite subset of $A$ or, equivalently, that the restriction (the ``mutilation'') of $t_{A}$ to a finite subset $A'$ must be exactly $t_{A'}$. This condition (called \emph{mutilation condition} in \cite{Girard1986}) says, intuitively, that the dependency of $t_{A}$ on $A$ is so flawed that one can compute $t_{A}$ starting from finite objects $A'$ and then extend by continuity. In particular, the family $(t_{A})_{A\in \B{Dom}}$, seen as an application from domains to morphisms between domains, cannot distinguish between distinct infinite domains.

Continuous families are very scarce: the mutilation condition dramatically collapses the degrees of freedom in the construction of the family $(t_{A})_{A\in \B{Dom}}$. For instance, there is exactly one continuous family $(t_{A})_{A\in \B{Dom}}$ (up to isomorphism) such that $t_{A}\in A\to A$: this is the identity family $(id_{A})_{A\in \B{Dom}}$, corresponding to the polymorphic identity $\TT{ID}$.

\subsection{Functors and natural transformations.} The mutilation condition can be seen as an example of a functorial interpretation of polymorphism. It is well-known that logical connectives correspond to functorial constructions (see \cite{Lambek1986, Dosen1999}). Usually, a construction
leading from objects $A_{1},\dots, A_{n}$ to an object $\kappa(A_{1},\dots, A_{n})$ in a given category $\B C$ is called functorial when it can be extended in a somehow canonical way into a construction $\kappa$ defined on all the objects in the category.

The canonicity of the extension is expressed as follows: let the object $\kappa(A_{1},\dots, A_{n})$ be
constructed from objects $A_{1},\dots, A_{n}$ and the object $\kappa(B_{1},\dots, B_{n})$ be constructed from
objects $B_{1},\dots, B_{n}$ such that all $B_{i}$ can be mapped to $A_{i}$ by some morphism $f_{i}$. Then the
two constructions can be related: a morphism $\kappa(f_{1},\dots, f_{n})$  from $\kappa(B_{1},\dots, B_{n})$ to $\kappa(A_{1},\dots, A_{n})$ can be constructed from the morphisms $f_{1},\dots, f_{n}$. In a sense, a functorial construction is so general that, when applied to arbitrarily related objects, it produces related objects\footnote{This is also the starting point of Reynolds' parametricity, based on an extension of the notion of functorial constructions from preservation of morphisms to preservation of binary relations (see \cite{Reynolds1983, Hermida2014}).
}. A functor $F (X)$ over a category $\B C$ can thus be defined as a map acting over the objects of $\B C$ and over its morphisms\footnote{More precisely, a covariant functor $ F (X)$ over a category $\B C$ is given by a map over the object of $\B C$ and a map over the morphisms of $\B C$, where for all objects $A,B$ and $f : A \to B, F (f) : F (A) \to F (B)$, preserving identity morphisms and composition, i.e. such that $F (id_{A}) = id_{F (A)}$ and $F (g)\circ F (f) = F (g\circ f)$ for all objects $A,B,C$ and morphisms $f:A\to B, g:B\to C$.
}.

The starting point of the functorial interpretation is the acknowledgement that most construction over types are functorial. For instance, the construction $F (X) = X \times X$ which, to any object $A$ of a category $C$, associates its product $A\times A$ is functorial: take another object $B$ and a morphism $f : A \to B$; then there exists a morphism $F (f) = f\times f$ which goes from $F (A) = A\times A$ to $F (B) = B \times B$.

In usual categorial interpretations of typed $\lambda$-calculi (and proof systems), types correspond to the objects of a given category $\B C$ and terms correspond to morphisms between such objects. One can think of a term of type $\sigma\to\tau$ as a morphism between the objects associated, respectively, to $\sigma$ and $\tau$.

A more abstract categorial interpretation arises then from the remark that a type $\sigma$, depending on a free variable $X$, corresponds to a functorial construction $F_{\sigma} (X)$ i.e. to a functor over the category $\B C$. As soon as types and become functors $F _{\sigma}(X)$ and $F_{\tau} (X)$, a term of type $\sigma\to \tau$ becomes a natural transformation\footnote{	In the general case, terms are actually interpreted by \emph{dinatural} transformations, a generalization of the notion of natural transformations in the case in which the functors $F (X)$ and $F (X)$ are multi-variant (as is the case when and contain an implication), see \cite{Bainbridge1990, Girard1992}.
} between $F_{\sigma}$ and $F_{\tau}$ , i.e. a family of morphisms $(t_{A})_{A\in Ob(\B C)}$ from $F_{\sigma} (A)$ to $F_{\tau} (A)$ satisfying a naturality condition of the form

$$
\xymatrix{
F_{\sigma}(A)   \ar[r]^{t_{A}} \ar[d]_{F_{\sigma}(f)}  & F_{\tau}(A) \ar[d]^{F_{\tau}(f)} \\
F_{\sigma}(B) \ar[r]_{t_{B}} & F_{\tau}(B)
}
$$

for all objects $A,B$ and morphism $f : A \to B$. The naturality condition warrants that the morphisms $t_{A}$ behave in the same way for all $A$. For instance, if we consider the case of a natural family $(d_{A})_{A\in Ob(\B C)}$ of morphism from the identity functor $X$ to the product
functor $X	\times X$, the naturality condition becomes:

$$
\xymatrix{
A \ar[r]^{d_{A}} \ar[d]_{f} & A\times A \ar[d]^{f\times f}  \\
B \ar[r]_{d_{B}} & B\times B
}
$$

The diagram above says that $d_{A}$ and any other $d_{B}$, with $B$ related to $A$ by a morphism $f$, are related by the equation $(f\times f) \circ d_{A} = d_{B}\circ f$. Intuitively, the only way to pass from $A$ to $A\times A$ without any information about $A$ is by duplication. Indeed there is exactly one family satisfying the naturality condition above, the diagonal family $d_{A}(x)=(x,x)$: natural families are, in a sense, so uniform that they have very few degrees of freedom.

The mutilation condition, which says that $t_{A}$ can be computed from the finite $t_{A'}$, with $A'$ related to $A$ by an inclusion morphism, can be seen as a special case of the naturality condition. In the category $\B{Dom}$ types $\sigma,\tau$ become continuous functors $F_{\sigma}, F_{\tau}$, i.e. functors whose values can be approximated from values on finite domains, and, inclusion morphisms can be inverted, one has the diagram below:
$$
\xymatrix{
F_{\sigma}(A) \ar[r]^{t_{A}} \ar[d]_{F_{\sigma}(\iota)^{-1}} & F_{\tau}(A) \\
F_{\sigma}(A') \ar[r]_{t_{A'}}  & F_{\tau}(A') \ar[u]_{F_{\tau}(\iota)}
}
$$
(where $\iota: A'\to A$ is the inclusion morphism) which says that $t_{A}$ can be computed starting from $t_{A'}$ and $\iota$.
The naturality condition appears as a very perspicuous characterization of the uniformity of polymorphic programs, as the latter should correspond to functions over types ``that behave in the same way for all types'' \cite{Reynolds1983}.

Functorial interpretations of typed $\lambda$-calculi and natural deduction (see \cite{Bainbridge1990, Girard1992}) show that terms (i.e. proofs) correspond to natural transformations, i.e. to uniform families. Actually, one can dare to say that proofs are natural transformations, as the functorial interpretation can be reproduced over the syntactic category generated by propositions and proofs themselves (this perspective, and some of its relevant consequences in proof theory, are investigated in \cite{StudiaLogica}, where a ``functorial'' presentation of natural deduction is provided).

\section{Predicates defined by recursion in $\B{ITT}_V$.}\label{app::B}
 We show how to define predicates by recursion in $\B{ITT}_V$ . This fact marks a basic difference with respect to $\B{PA}^2$. As discussed in subsection \ref{subs::3.3}, in $\B{PA}^2$one cannot define a predicate $\TT{Comp}(x, y)$ such that, for each formula $A \in \C L^{2}$, the predicate $\TT{Comp}_{A}(y)$ is equivalent to $\TT{Comp}(\sharp A, y)$. On the contrary, by defining predicates by recursion over the codes of formulas and by using standard recursive techniques, the stipulations given in subsection \ref{subs::3.3} can be turned into a definition in $\B{ITT}_V$ of the following:
 \begin{itemize}

\item a function $\tau\in (\Pi x\in \B N) \TT{FV}(x)\to V$, where $\TT{FV}\in \B N \to V$ is a predicate such that, for each type $A(x_{1},\dots, x_{n})\in V$, for $x_{1}\in A_{1}$, \dots $x_{n}\in A_{n}(x_{1},\dots, x_{n-1})$, $\TT{FV}(\sharp A)= \tau(\sharp A_{1})\times \dots \times \tau(\sharp A_{n})$ (and $\TT{FV}(\sharp a)= \B 1= (\Pi x\in V)(x\to x)$ if $a$ has no free variable), $\tau(\sharp A)\in \TT{FV}(\sharp A)\to V$ and for all $(\xi_{1},\dots, \xi_{n}\in \TT{FV}(\sharp A)$, $\tau(\sharp A)(\xi_{1},\dots, \xi_{n})\in V$ represents the set $s_{A}(\xi_{1},\dots, \xi_{n})$;

\item a predicate $\TT{Comp}\in (\Pi x\in \B N)(\TT{FV}(x)\times \B N\times \tau(x)\to V)$ such that, for each type $A(x_{1},\dots, x_{n})\in V$, where $x_{1}\in A_{1},\dots, x_{n}\in A_{n}(x_{1},\dots, x_{n-1})$, $\TT{Comp}(\sharp A)\in \TT{FV}(\sharp A)\times \B N\times \tau(\sharp A)\to V$ is a predicate $\B N$ over $\tau_{A}$ (with parameters in $\TT{FV}(\sharp A)$ corresponding to $\TT{Comp}_{A}$;

\item a function $\boldsymbol \alpha\in (\Pi x\in \B N)(\Pi z\in \TT{FV}(x))\tau_{z}(\TT{typ}(x))$ where

\begin{itemize}

\item $\TT{FV}(\sharp a(x_{1},\dots, x_{n}))$, for $x_{1}\in A_{1},\dots, x_{n}\in A_{n}(x_{1},\dots, x_{n-1})$, is $\tau(\sharp A_{1})\times \dots \times \tau(\sharp A_{n})$;
\item $\TT{typ}\in \B N\to \B N$ is a primitive recursive predicate such that, for all $a\in A$, $\TT{typ}(\sharp a)=\sharp A$;

\item $\tau_{z}(\sharp a)$ is $(\lambda u\in \TT{FV}(\sharp a))\tau(\sharp a)\big (\pi(1)(q_{1})\dots \pi(m)(q_{m})\big )\in \TT{FV}(\sharp a)\to V$, where $m$ is the number of free occurrences of $A$ and $q_{i}$ is either $z$ or $u$ depending on whether $x_{i}$ occurs free in both $a$ or only in $A$ (the primitive recursive function $\tau_{z}$ can be constructed by standard coding techniques).

\end{itemize}
$\boldsymbol \alpha$ is such that, for each closed term $a\in A$, $\boldsymbol \alpha(\sharp a)\in \tau(\sharp A)$ represents $\alpha_{a}$.
\end{itemize}

The definition of predicates by recursion is showed by the two propositions below.

\begin{proposition}\label{flat}
Let $k_{1},\dots, k_{m}\in \B N \to \B N$ be functions such that, for all $x\in \B N$, $k_{i}(x+1)< x+1$. 
Then, then, for all type $A$, $g\in\B N^{p}\to A$ and $h\in \B N^{p+m+1}\to A$, there exists a function $f\in \B N^{p+1}\to A$ such that

\begin{equation*}\small
\begin{split}
f(x_{1},\dots, x_{n},0)  &= g(x_{1},\dots, x_{n}) \\
f(x_{1},\dots, x_{n},y+1) & = h(x_{1},\dots, x_{n}, y+1, f(x_{1},\dots, x_{n},k_{1}(y+1)), \dots, f(x_{1},\dots, x_{n},k_{m}(y+1)))
\end{split}
\end{equation*}

\end{proposition}
\begin{proof}
Let $F(\vec x,y)=\Pi_{z\leq y}\pi(y)^{\sharp f(\vec x, z)}\in \B N^{p+1}\to \TT{List}[A]$, where $\pi(x)$ is the recursive function which enumerates primes and
$$
 \TT{List}[X]= (\Pi x\in V)(x \to (X \to x\to x)\to x)
 $$
 is the type of lists of type $X$,
be the function defined by primitive recursion as follows:

\begin{equation*}
\begin{split}
 F(\vec x, 0) & = g(\vec x) \\
F(\vec x, y+1) & = F(\vec x,y) \ \TT{conc} \  h(\vec x, \delta_{V}(k_{1}(y), F(\vec x,y)), \dots, \delta_{V}(k_{m}(y), F(\vec x,y)))
\end{split}
\end{equation*}
Where $\delta_{X}\in\B N \to \TT{List}[X]\to X $ is the primitive recursive function such that \\ $\delta_{X}(i, \langle x_{1},\dots, x_{n}\rangle)=x_{i}$, for $i\leq n$.
Then we can put $f(\vec x, y)=\delta_{A}(y, F(\vec x, y))\in \B N^{p+1}\to A$.

\end{proof}

The construction of this proposition enables, for instance, the construction of a function $\flat:\B N\to V$ such that $\flat\sharp A=A$. The definability of $\flat$ marks a sensitive difference with System $\B F$ (and $\B{PA}^{2}$): in those systems there's no way to pass from the code of a type (or of a proposition) to the type (or proposition) itself.

A more general result can be shown with the aid of the $\flat$ function:

\begin{proposition}\label{rec}
Let $k_{1},\dots, k_{m}\in \B N \to \B N$ be functions such that, for all $x\in \B N$, $k_{i}(x+1)< x+1$. 
Let $C(x_{1},\dots, x_{n})$ and $D(x_{1},\dots, x_{n},y,y_{1},\dots, y_{n})$ be types.
Then, for all $g\in (\Pi x_{1},\dots, x_{n}\in \B N)C(x_{1},\dots, x_{n})$\footnote{Where $(\Pi x_{1},\dots, x_{n}\in A)B(x_{1},\dots, x_{n})$ denotes the product $(\Pi x_{1}\in A)\dots (\Pi x_{n}\in A)B(x_{1},\dots, x_{n})$.} and $h\in (\Pi x_{1},\dots, x_{n},y, y_{1},\dots, y_{m}\in \B N)D(x_{1},\dots, x_{n}, y, y_{1},\dots, y_{m})$, there exists a type $A(x_{1},\dots, x_{n},y)$ and a function $f\in (\Pi x_{1},\dots, x_{n},y\in \B N)A(x_{1},\dots, x_{n},y)$ such that

\begin{equation*}\small
\begin{split}
A(x_{1},\dots, x_{n},0) &=C(x_{1},\dots, x_{n}) \\
 f(x_{1},\dots, x_{n},0)  &= g(x_{1},\dots, x_{n}) \\
 A(x_{1},\dots, x_{n},y+1)&= D(x_{1},\dots, x_{n},y+1,f(x_{1},\dots, x_{n},k_{1}(y+1)), \dots, f(x_{1},\dots, x_{n},k_{m}(y+1)))\\
 f(x_{1},\dots, x_{n},y+1)  &= h(x_{1},\dots, x_{n}, y+1, f(x_{1},\dots, x_{n},k_{1}(y+1)), \dots, f(x_{1},\dots, x_{n},k_{m}(y+1)))
 \end{split}
\end{equation*}

\end{proposition}
\begin{proof}
Let $T(\vec x,y)= \Pi_{z\leq y}\pi(y)^{\sharp A(\vec x, z)}$ and $F(\vec x,y)=\Pi_{z\leq y}\pi(y)^{\sharp f(\vec x, z)}$ which are defined by mutual recursion as follows:

\begin{equation*}
\begin{split}
 T(\vec x, 0) &= 2^{\sharp C(\vec x)} \\
 F(\vec x, 0)&=2^{\sharp g(\vec x)} \\
 T(\vec x, y+1)&= T(\vec x, y) \cdot \pi(y+1)^{\sharp D(\vec x, \delta(k_{1}(y),\flat F(\vec x,y)), \dots, \delta(k_{m}(y), F(\vec x,y)))}\\
 F(\vec x, y+1)&= F(\vec x,y) \cdot \pi(y+1)^{\sharp h(\vec x, \delta(k_{1}(y),\flat F(\vec x,y)), \dots, \delta(k_{m}(y), F(\vec x,y)))}
 \end{split}
\end{equation*}
Where $\delta\in \B N^{2}\to \B N$ is the primitive recursive function such that $\delta(i, \sharp (x_{1},\dots, x_{n}))=x_{i}$, for $i\leq n$.
Then we can put $A(\vec x, y)= \flat \delta(y, T(\vec x, y))$ and $f(\vec x, y)=\flat \delta(y, F(\vec x, y))$.

\end{proof}

\begin{landscape}

\begin{figure}\label{fig::1}
\vspace{57mm}
\resizebox{1.6\textwidth}{!}{$$
\tiny
\AXC{$a\in  (\Pi m\in \B N)(\Pi w \in \tau_{V})(\TT{Comp}_{V}(m, w) \Rightarrow \TT{Comp}_{B(w)}(\TT{App}(n, m), v w))$}
\AXC{$\sharp C\in \B N$}
\RightLabel{$\Pi E$}
\BIC{$a\sharp C\in  (\Pi w \in \tau_{V})(\TT{Comp}_{V}(\sharp C, w) \Rightarrow \TT{Comp}_{B(w)}(\TT{App}(n, \sharp C), vw))$}
\AXC{$(\boldsymbol\alpha_{C}, \TT{Comp}_{C})\in \tau_{V}$}
\RightLabel{$\Pi E$}
\BIC{$(a\sharp C)(\boldsymbol\alpha_{C}, \TT{Comp}_{C})\in(\TT{Comp}_{V}(\sharp C, (\boldsymbol\alpha_{C}, \TT{Comp}_{C})) \Rightarrow \TT{Comp}_{B(\TT{Comp}_{C})}(\TT{App}(n, \sharp C), v(\boldsymbol\alpha_{C}, \TT{Comp}_{C})))$}
\AXC{$b\in \TT{Comp}_{V}(\sharp C, (\boldsymbol\alpha_{C},\TT{Comp}_{C}))$}
\RightLabel{$\Pi E$}
\BIC{$ ((a\sharp C) (\boldsymbol\alpha_{C}, \TT{Comp}_{C}))b\in\TT{Comp}_{B}((\boldsymbol\alpha_{C}, \TT{Comp}_{C}))(\TT{App}(n, \sharp C), v(\boldsymbol\alpha_{C}, \TT{Comp}_{C}))$}
\DP
$$
}
\caption{}\label{1}
\end{figure}
\end{landscape}

\end{document}